\documentclass{amsart}
\usepackage{graphicx}
\usepackage{amsmath, amssymb, comment, hyperref, amsfonts, amsthm, mathtools}
\usepackage{adjustbox, caption, subcaption}

\usepackage{xcolor,tikz-cd}

\usepackage[capitalise]{cleveref}

\newcommand{\C}{\mathcal{C}}
\newcommand{\Dm}{\mathcal{C}_m}
\newcommand{\Dpm}{\Dm'}

\newcommand{\Dmc}{\mathcal{C}_m^\circlearrowright}
\newcommand{\ipo}{{i+1\;(\mathrm{mod}\;m)}}
\newcommand{\Dim}{\mathsf{D}_m}

\newcommand{\deldm}{\delta^{\Dm}}
\newcommand{\deldpm}{\delta^{\Dpm}}
\newcommand{\delddm}{\delta^{\Dm,\Dpm}}
\renewcommand{\P}{\bigstar}

\newcommand{\Set}{{\mathsf{Set}}}
\newcommand{\Top}{{\mathsf{Top}}}
\newcommand{\Ch}{{\widehat{\C}}}
\newcommand{\E}{\mathcal{E}}
\newcommand{\G}{\mathcal{G}}
\DeclareMathOperator{\Hom}{Hom}

\newtheorem{theorem}{Theorem}[section]
\newtheorem{lemma}[theorem]{Lemma}
\newtheorem{cor}[theorem]{Corollary}
\newtheorem{prop}[theorem]{Proposition}

\theoremstyle{definition}
\newtheorem{defi}[theorem]{Definition}
\theoremstyle{remark}
\newtheorem{ex}[theorem]{Example}
\newtheorem{remark}[theorem]{Remark}

\title[Categorical Tiling Theory]{Categorical Tiling Theory: Constructing Directed Planar Tilings via Edge Reversal}
\author{Catherine DiLeo \and Preston Sessoms \and Brandon T. Shapiro}

\begin{document}

\maketitle

\begin{abstract}
    Tilings of the plane resemble the simplicial and other complexes from algebraic topology, but have not been studied from this perspective. We construct finite categories corresponding to polygons with labeled directed edges, and introduce the problem of modeling tilings of the Euclidean or hyperbolic plane as presheaves over such a category. Combinatorially, this amounts to choosing an ``alignment'' for a tiling: a direction for every edge and consistent labels for the edges of each polygonal tile. We show that for a fixed tiling, given a single alignment we can characterize every other alignment of the same tiling by comparison of the edge directions. We then construct a ``reflective'' alignment for any tiling with an even number of polygons at each vertex, and from this generate a large family of alignments with elegant symmetry properties.
\end{abstract}

\tableofcontents

\section*{Introduction}

In the classical theory of regular tilings, there are exactly three ways to divide the Euclidean plane into congruent regular polygons which meet at common edges and vertices: 6 triangles meeting at each point, 4 squares at each point, or 3 hexagons at each point. These three tilings, denoted by their Schl\"afli symbols $\{3,6\}$, $\{4,4\}$, and $\{6,3\}$, are pictured in Figure~\ref{introeuclideantilings}, 
\begin{figure}[h]
    \includegraphics[width=2.5cm]{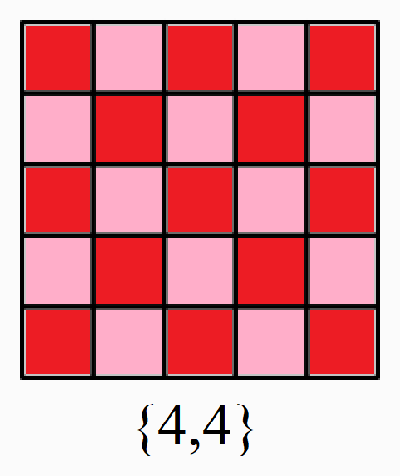}
    \qquad\quad
    \includegraphics[width=3cm]{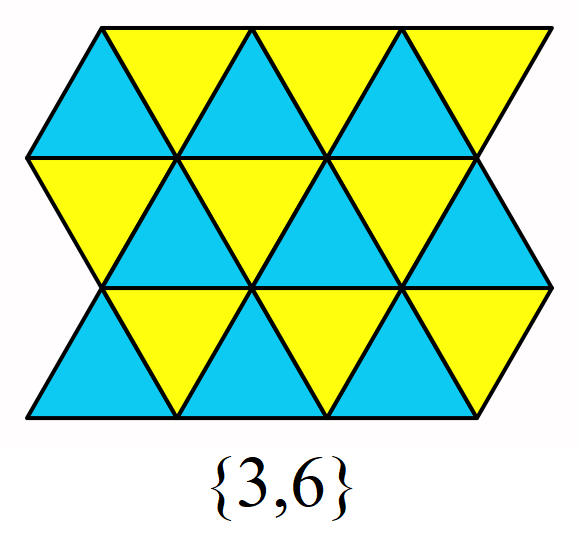}
    \qquad\quad
    \includegraphics[width=3cm]{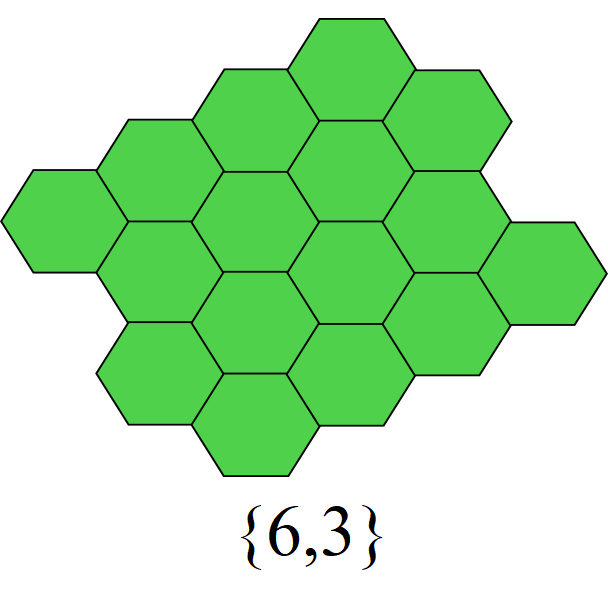}
    \caption{The 3 regular tilings of the Euclidean plane.}
    \label{introeuclideantilings}
\end{figure}
where $\{m,n\}$ denotes an arrangement of $n$ different $m$-sided polygons ($m,n \ge 3$). The symbols $\{3,3\}$, $\{4,3\}$, $\{3,4\}$, $\{5,3\}$, and $\{3,5\}$ describe the 5 regular polyhedra, which can be thought of as tilings of the sphere, while any symbol $\{m,n\}$ with higher numbers than these 8 describes a tiling of the hyperbolic plane such as those in Figure~\ref{introhyperbolictilings}.

\begin{figure}[h]
\qquad\qquad
    \includegraphics[width=2.8cm]{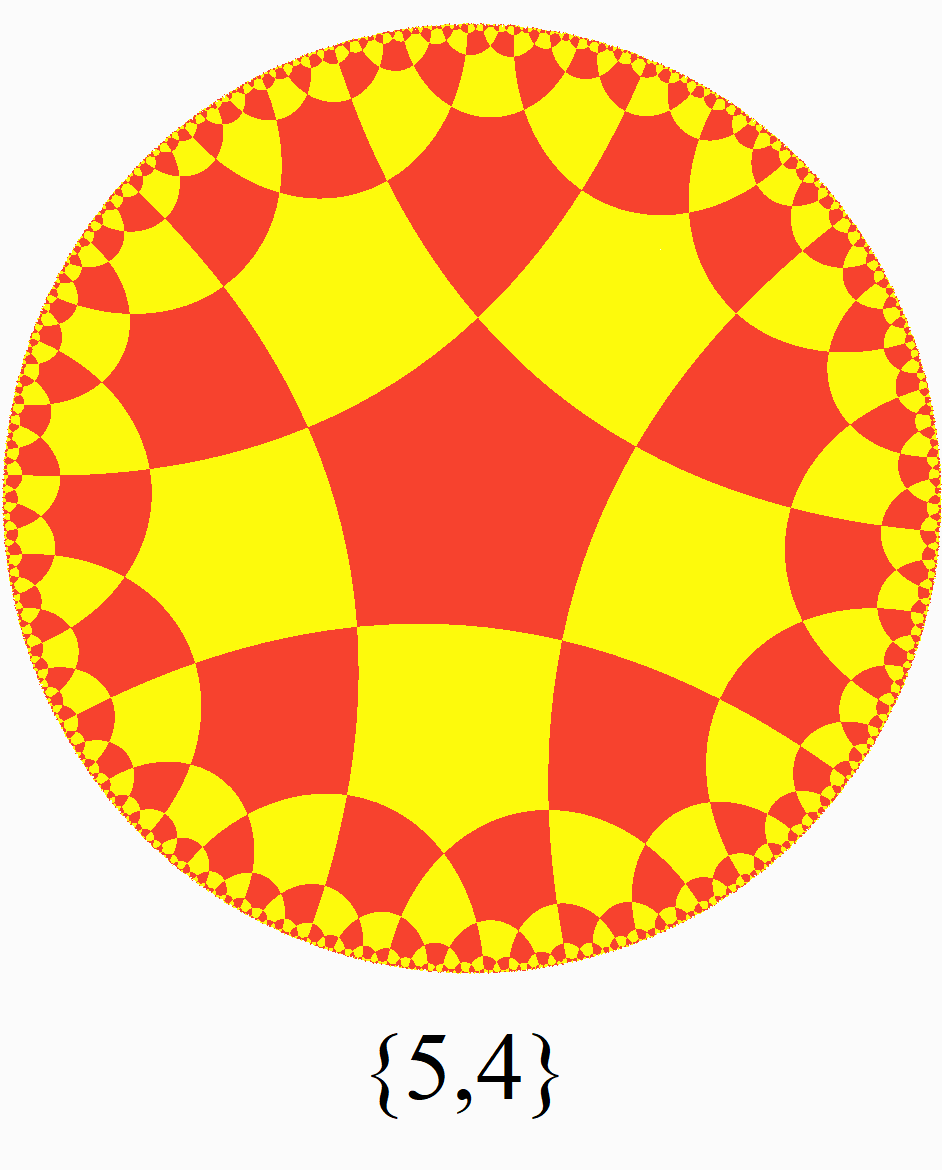}
    \qquad\quad
    \includegraphics[width=2.8cm]{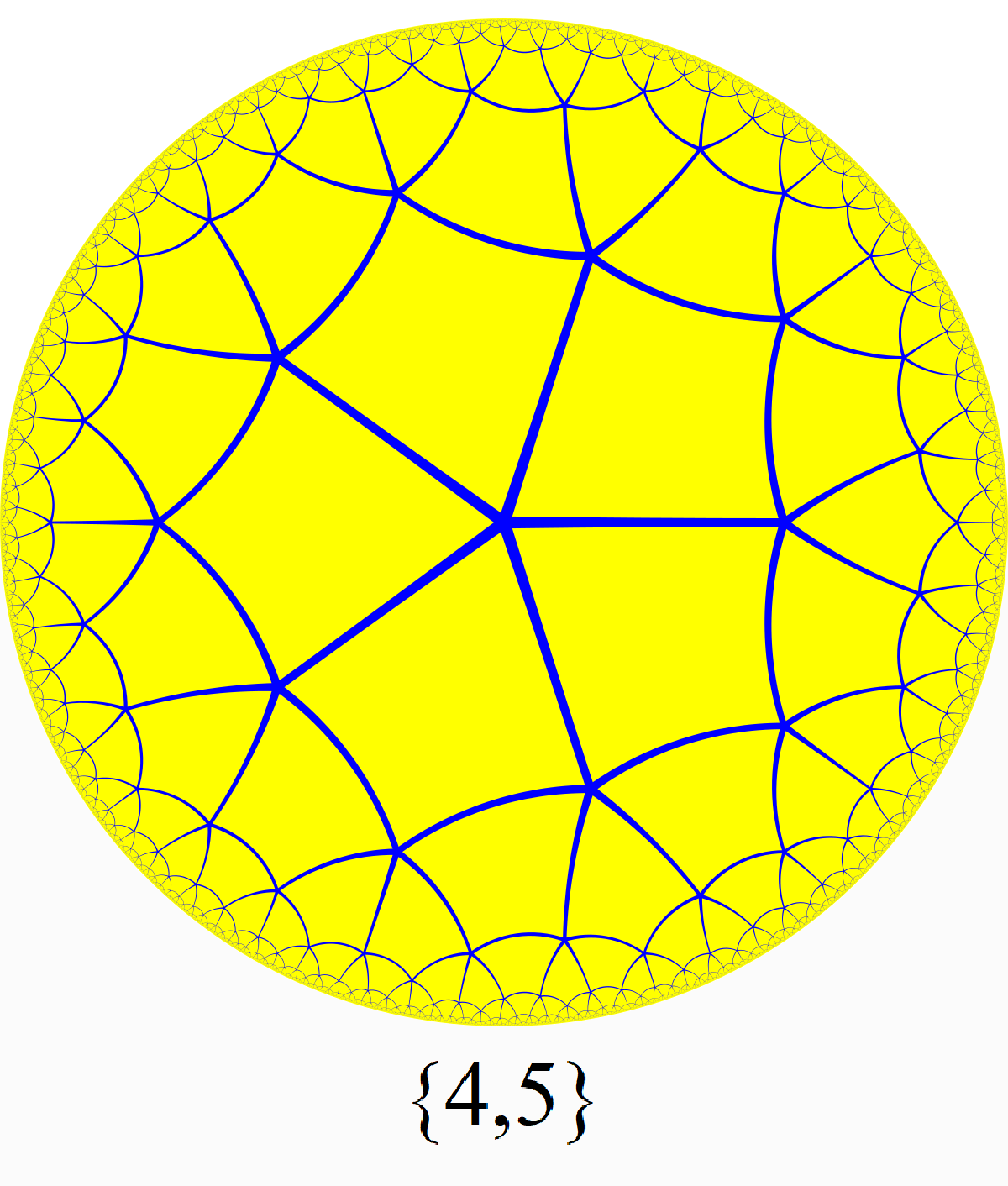}
    \qquad\quad
    \includegraphics[width=2.8cm]{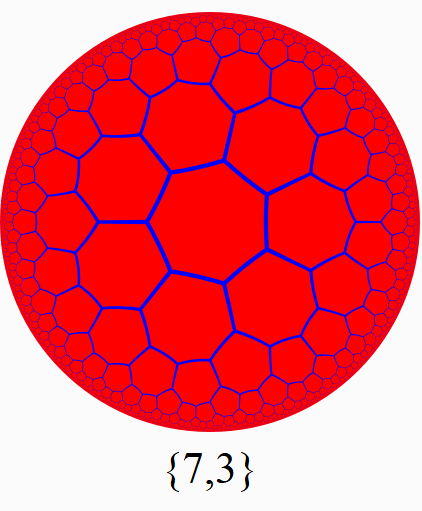}
    \caption{Some regular tilings of the hyperbolic plane, pictured in the Poincare disk model \cite{tiling54,tiling45,tiling73}}
    \label{introhyperbolictilings}
\end{figure}

Covering a space such as the plane or sphere with fixed polygonal shapes in this manner is common not only in classical geometry but also in topology, where arrangements of polygons and higher dimensional shapes are used to model spaces. A triangulation of a space, typically a manifold, is a construction of the space out of vertices, edges, triangles, and higher dimensional simplices as needed. The combinatorial data of how these shapes are connected contains all of the relevant topological information but is easier to analyze than a topological space. Spaces can similarly be constructed out of other shapes such as squares and cubes.

In algebraic topology, these combinatorial models of spaces are particularly convenient for computing (co)homology (for a basic example, see \cite[Example 3.7]{hatcher}) and studying homotopy theory. 
While there are many formal models for combinatorial spaces, those based on the notion of \emph{presheaves} from category theory are most commonly used for their elegant and convenient properties. This includes simplicial sets, semisimplicial sets (sometimes called $\Delta$-complexes), and cubical sets, among others.

A presheaf model of a space is based on a collection of shapes defined by their \emph{faces}, which are labeled pieces of the shape's boundary given by other shapes in the collection. In the 2-dimensional simplicial case, the shapes are the vertex, edge, and triangle, with faces labeled as in Figure~\ref{fig:simplexcat}. Often other relationships between shapes such as projections are considered to produce ``degenerate'' shapes, but we will only need faces to describe tilings.

\begin{figure}
    \centering
    \includegraphics[width=0.6\linewidth]{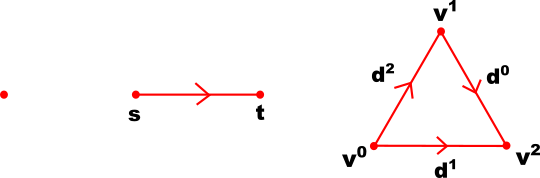}
    \caption{The vertex, edge, and triangle with labeled faces. The arrow on the edge always points from the $s$ vertex (for ``source'') to the $t$ vertex (for ``target''), as a substitute for including the labels as in the triangle.}
    \label{fig:simplexcat}
\end{figure}


A presheaf $X$ on these shapes consists of sets $X_c$ of cells of each shape $c$ in the collection, where whenever the shape $c$ has a face $f$ of shape $c'$, there is a function $X_c \to X_{c'}$ sending each $c$-shaped cell in the set $X_c$ to its $f$-labeled face, a $c'$-shaped cell in the set $X_{c'}$. This information can be packaged as a category $\C$ with shapes $c$ as objects and a morphism $f \colon c' \to c$ for each $c'$-shaped face of the shape $c$, where a presheaf is a contravariant functor $X \colon \C^{op} \to \Set$ and a map of presheaves is a natural transformation.

A presheaf models a space when its cells, realized as the topological shapes they represent and glued together along shared faces, assemble to form that space. In Figure~\ref{fig:intropresheaf} for instance, a presheaf with vertices $\{a,b,c,d,e\}$, edges $\{\alpha,\beta,\gamma,\delta,\epsilon,\theta,\omega\}$, and triangles $\{A,B\}$ models a space built by gluing together two triangles and an edge ($\omega$) along vertices, homotopy equivalent to a circle. The relationships between these shapes in the associated space can be expressed by equations of the form $v_1(A)=v_0(B)=b$ in the presheaf.

\begin{figure}
    \centering
    \includegraphics[width=0.4\linewidth]{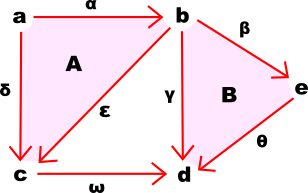}
    \caption{A picture of a presheaf with 5 vertices, 7 edges, and 2 triangles.}
    \label{fig:intropresheaf}
\end{figure}

While planar tilings clearly fit the pattern of spaces built out of fixed cell shapes, they have not been studied as presheaves. This may be because to do so requires the edges of the polygons and the endpoints of the edges to be labeled. Whereas the vertices and edges of a tiling form an undirected graph, a presheaf model of a tiling would have sets of vertices and edges which form a directed graph as each edge has a specified source and target. Choosing directions for these edges constitutes a massive amount of new information, and the options are limited by the requirement that each polygon fits the pattern of edge directions specified by the category of cell shapes. 

With the goal of bridging the gap between the geometric/combinatorial theory of tilings and the categorical techniques of algebraic topology, we introduce definitions, examples, and basic results on these presheaf models of regular planar tilings, which we call \emph{directed tilings}. We begin by analyzing the \emph{directed tiles}, namely $m$-gons with fixed patterns of edge directions such as the directed pentagons in Figure~\ref{fig:intropentagons}, by means of their corresponding \emph{$m$-gon categories}. 

\begin{figure}[h]
    \centering
    \includegraphics[scale=.4]{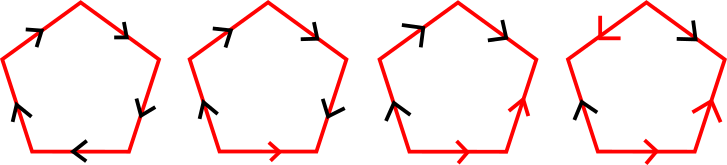}
    \caption{Four different directed pentagon tiles.}
    \label{fig:intropentagons}
\end{figure}

We then define directed tilings and proceed to analyze the possible compatible choices of edge directions and labels for a fixed $\{m,n\}$ tiling, data we call an \emph{alignment}. In particular, in \cref{newtiling} we show that given a single directed tiling $T$ every possible alignment of the same underlying tiling is determined by a map of presheaves from $T$ to a presheaf $\Lambda$ satisfying certain conditions. This map represents a choice for each edge in the tiling of whether to preserve or reverse its direction from the alignment of $T$.

Finally, we construct a large family of examples of directed tilings with $4k$ tiles around each vertex whose alignments look the same within a neighborhood of every tile. This construction uses the process of finding each tile in the plane by a sequence of reflections from a chosen ``base'' tile, where each reflection modifies the edge directions based on the label of the edge a tile is reflected across. In \cref{thm:generatetiling} we show that this process is sufficient to specify an alignment of an entire tiling, and in \cref{automorphismphi} we show that for these \emph{reflection generated} directed tilings any transformation of the plane generated by reflections across edges in the tiling modifies the edge directions of every tile in exactly the same way. As such, reflection generated directed tilings such as those in Figure~\ref{fig:introdirectedtilings} appear highly symmetric.

\begin{figure}[h]
    \begin{minipage}{.5\textwidth}
        \centering
        \includegraphics[width=.7\textwidth]{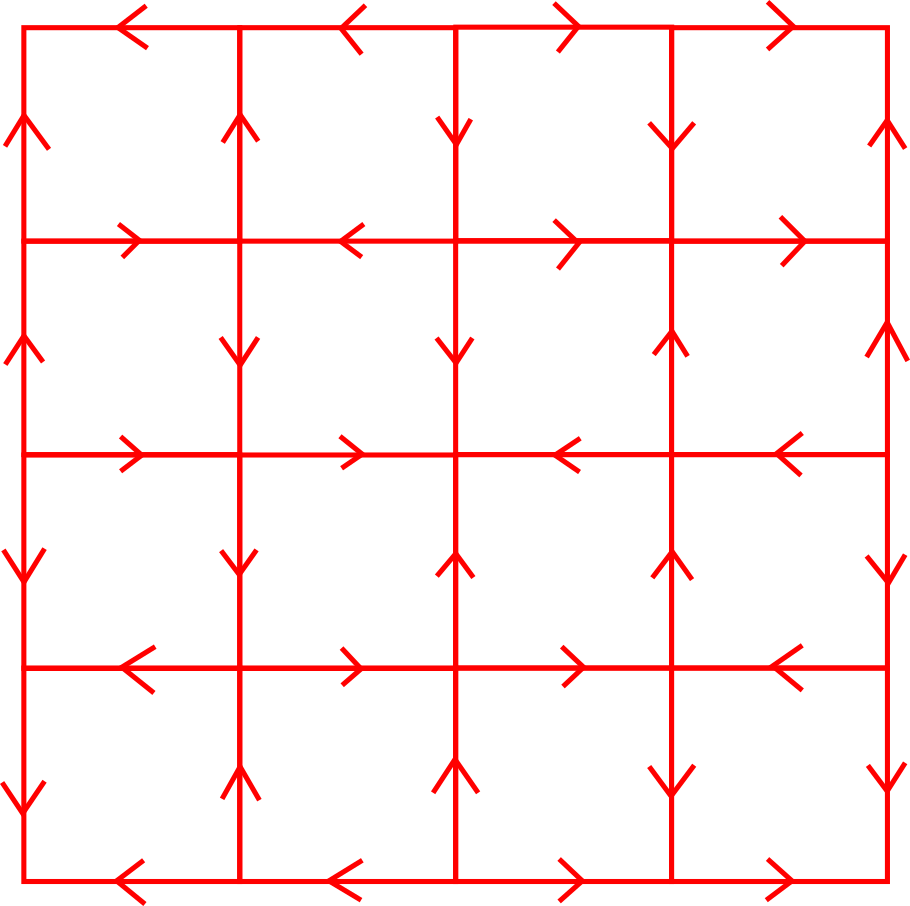}
    \end{minipage}%
    \begin{minipage}{.5\textwidth}
        \centering
        \includegraphics[width=.7\textwidth]{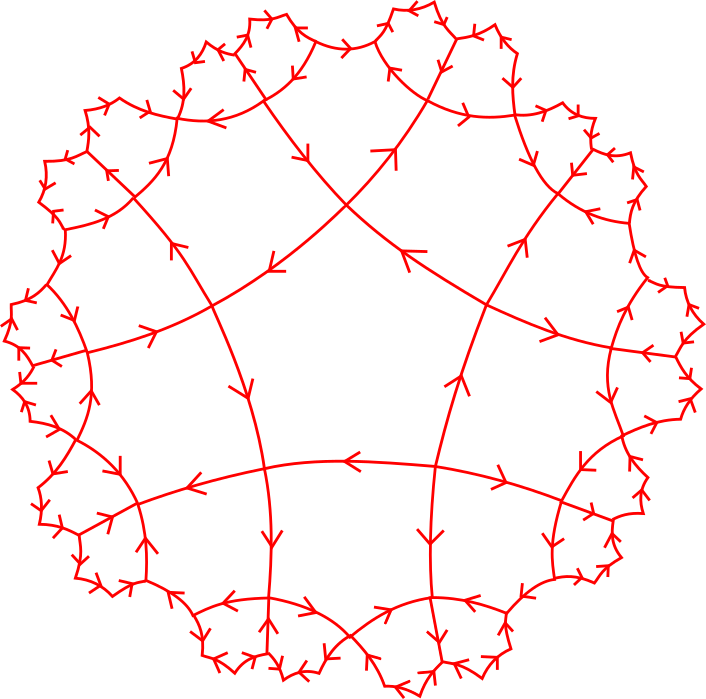}
    \end{minipage}%
    \caption{Two reflection generated directed tilings.}
    \label{fig:introdirectedtilings}
\end{figure}

\subsection*{Organization}

In Section 1 we review background on tilings, Coxeter groups, and presheaves. In Section 2 we introduce $m$-gon categories to represent directed tiles, count the number of $m$-gon categories up to isomorphism, and define directed tilings as certain presheaves on an $m$-gon category. In Section 3 we characterize all alignments for any undirected tiling which admits an alignment. In Section 4 we construct reflective and reflection generated tilings, and in Section 5 we prove symmetry properties of these tilings.

\subsection*{Acknowledgements}

This work was partially supported by NSF grant DMS-1839968, and originated in the 2024 UVa Topology REU. We would like to thank Paul Bianco and Ansel Huffman for working closely with us during the early stages of this project, along with Valentina Zapata Castro and David Chasteen-Boyd for answering countless questions. We also thank Thomas Koberda and Slava Krushkal for invaluable organizational support during the REU, as well as Julia Bergner, Sarah Blackwell, Beth Branman, Mark Pengitore, Ryan Stees, Sara Maloni, and Ben Spitz for helpful conversations. Finally, we thank the anonymous referee for numerous helpful suggestions.

\section{Background}

\subsection{Tilings}


Tilings of the Euclidean or hyperbolic plane are longstanding objects of study in mathematics, and we summarize some relevant aspects of this theory. More comprehensive references can be found in \cite[Section 6.7]{coxeterpolytopes}, \cite[Section 11.7]{coxetergeometry}, and \cite[Chapters 5-6]{conwaysymmetries}, among others.

\begin{defi}
    An $\{m,n\}$ tiling of a topological space for $m,n \ge 3$ is a decomposition of that space into $m$-gons (embedded closed disks called \emph{tiles}) which intersect at edges and/or vertices with $n$ tiles sharing each vertex. 
    A geometric $\{m,n\}$ tiling of a metric space has the additional property that the edges are all the same length and the angles of each $m$-gon are $\frac{2\pi}{n}$.
\end{defi}

One can alternatively build an $\{m,n\}$ tiling with an appropriate gluing construction, beginning with an infinite disjoint union of $m$-gons, and gluing along appropriate edges. This more combinatorial approach is closer to how we will model tilings categorically, but traditionally less common. 

For each pair $\{m,n\}$, there is a canonical geometric $\{m,n\}$ tiling of either the sphere ($\{3,3\},\{3,4\},\{3,5\},\{4,3\},\{5,3\}$), the Euclidean plane ($\{3,6\},\{4,4\},\{6,3\}$), or the hyperbolic plane (all other pairs). 

 \begin{figure}[h]
    \includegraphics[width=4cm]{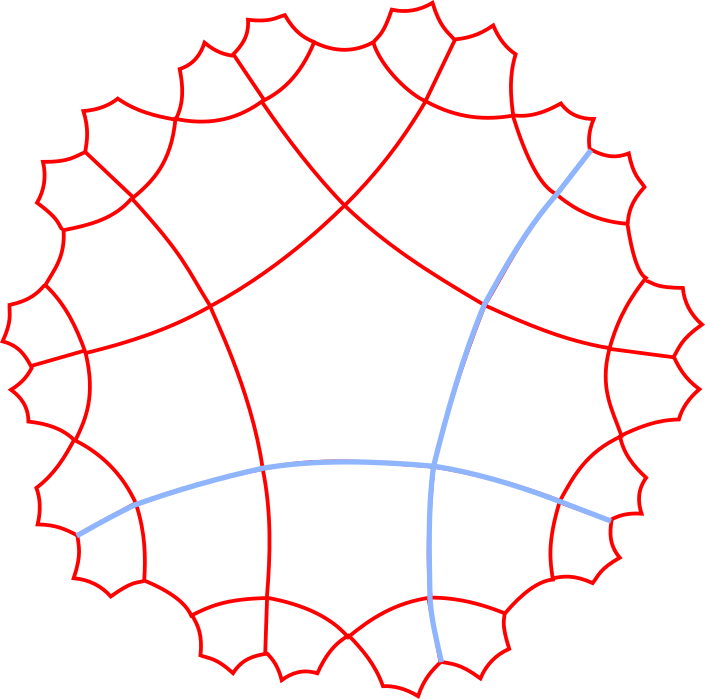}
    \qquad\qquad\qquad
    \includegraphics[width=4cm]{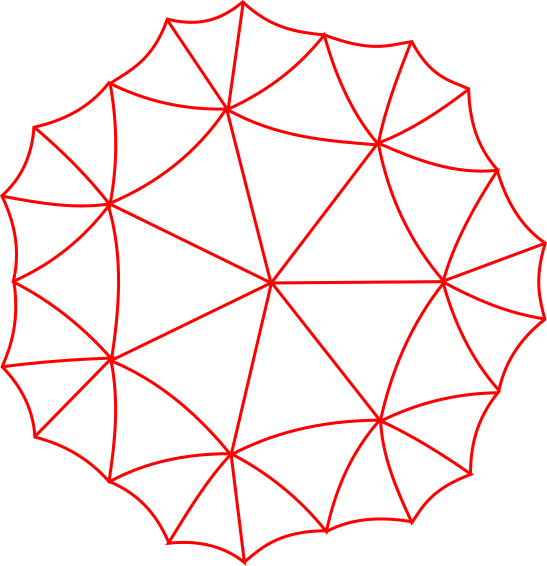}
    \caption{Portions of the $\{5,4\}$ and $\{3,7\}$ hyperbolic tilings, with geodesics shown in blue.}
    \label{hyperbolictilings}
    \end{figure}

Informally, a geodesic is a shortest path between two points. In the Poincare disc model of the hyperbolic plane, geodesics are arcs meeting the boundary at infinity at a right angle such as in Figure~\ref{hyperbolictilings}. 
In the Euclidean case, we inherit the regular metric on $\mathbb{R}^2$ and see that geodesics are straight lines. 
In $\{4,4\}$ (the square grid of Figure~\ref{euclideantilings}), for example, we see that some geodesics line up with paths of edges in the tiling. This is not the case in $\{6,3\}$ (honeycomb tiling), for example, where there is no straight line consisting of only edges of the tiling. This property depends on the parity of $n$. In tilings where $n$ is even, we can give a more combinatorial definition of geodesics that coincide with edges in the tiling.

\begin{figure}[h]
    \includegraphics[width=4cm]{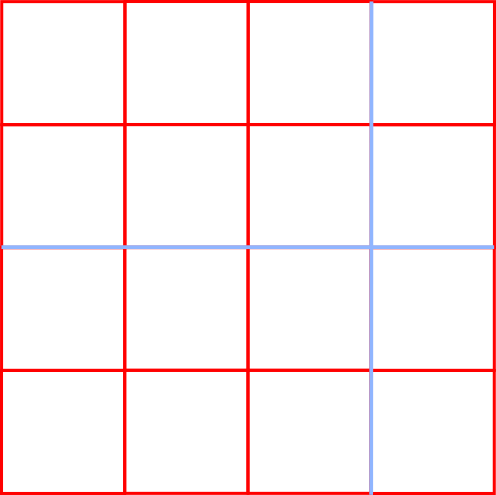}
    \qquad\qquad\qquad
    \includegraphics[width=4cm]{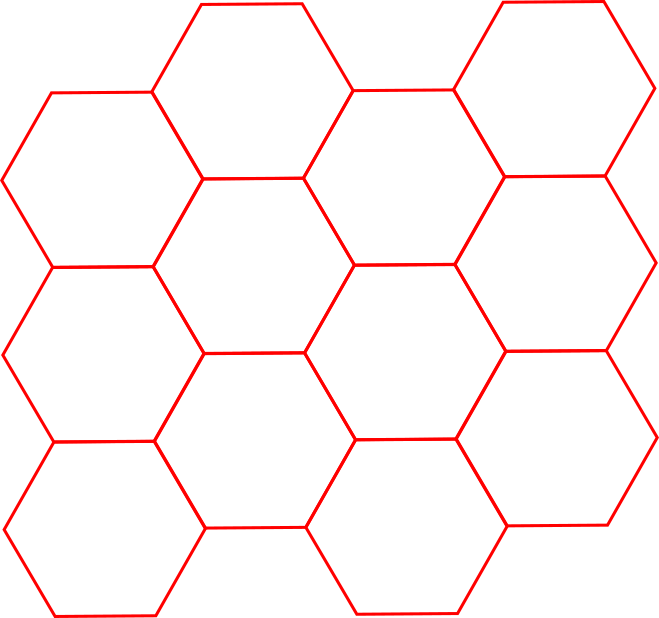}
    \caption{Portions of the $\{4,4\}$ and $\{6,3\}$ Euclidean tilings, with geodesics shown in blue.}
    \label{euclideantilings}
\end{figure}

\begin{defi}
    In an $\{m,n\}$ tiling where $n$ is even, a geodesic of the tiling is a sequence of vertices $v_i$ and edges $e_i$ where $i$ ranges over the integers, such that the edge $e_i$ goes between the vertices $v_{i-1}$ and $v_i$ and the edges $e_i$ and $e_{i+1}$ are opposite edges at the vertex $v_i$.
\end{defi}

In a geometric tiling, the geodesics of the tiling will also be geodesics in the ambient space as opposite edges at a vertex form a straight line in any 2-dimensional geometry. A reflection across a geodesic is an isometry of the Euclidean/hyperbolic plane. Reflecting across a geodesic of a tiling additionally sends vertices to vertices, edges to edges, and tiles to tiles, fixing only the vertices and edges on the geodesic.

In this paper, we are interested only in the combinatorial properties of a tiling rather than its metric structure, 
however we mention the geometry to establish an intuitive idea of how to visualize these tilings and their geodesics, and how they will be portrayed in figures.

\begin{defi}
    Two tiles are called \emph{adjacent} if they share a common edge. A sequence of adjacent tiles is called a \emph{track}. 
\end{defi}

\subsection{Coxeter Groups}

Coxeter groups are a large class of groups studied in the fields of algebra, geometry, and combinatorics, which include familiar examples such as the symmetry groups $\mathsf{S}_m$ and the dihedral groups $\Dim$. What is special about Coxeter groups is that they can be realized as a group generated by reflection transformations of some geometric object. 
The motivating example of Coxeter groups for our work arises from regular tilings of the Euclidean and hyperbolic plane.

\begin{defi}
    Let $S$ be a set, then a matrix $M: S \times S \to \{1,2,\cdots \infty\}$ is a Coxeter matrix provided
    \begin{itemize}
        \item $M(s,s') = M(s',s)$
        \item $M(s,s') = 1 \text{ if and only if } s=s'$
    \end{itemize}
\end{defi}

A Coxeter matrix can also be interpreted as an undirected graph with vertices the elements of $S$ and an edge between $s,s'$ if $M(s,s') \geq 3$. If $M(s,s') = 3$ this edge is unlabeled, otherwise, the edge is labeled with $M(s,s')$. Note that $M(s,s')$ is never 1 for distinct $s,s'$ by the second axiom, so if there is no edge between distinct $s,s'$ it means $M(s,s')=2$. This information provides the data required to generate a Coxeter group.

\begin{defi}
    A Coxeter group $W$ is a group generated by a set $S$ with generating relations of the form $(ss')^{M(s,s')} = e$ for a Coxeter matrix $M$. We say $(W,S)$ is a Coxeter system and $S$ contains the Coxeter generators.  
\end{defi}

Note that every generator has order 2, and for generators $s,s'$ with no edge between them in the ``Coxeter graph'' corresponding to $M$, $s,s'$ commute in $W$. The label on an edge between $s,s'$ thus denotes essentially ``how non-commutative'' they are in terms of the order of their product $ss'$.

Coxeter groups can be used to define vertex and face transitive tilings as an action on Euclidean and Hyperbolic space, $\mathbb{E}^2$ and $\mathbb{H}^2$. 

\begin{prop}[See for instance {\cite[Proposition 2.3.23]{notesoncoxetergps}, where $m$ and $n$ as used here correspond to $n$ and $2m$ respectively}]\label{mncoxeter}
    Let $m \geq 3$ and $n$ be even with $n > \frac{2m}{m-2}$. The Coxeter group $W_{m,n}$ with presentation graph an $m$–cycle and all edges labeled $n/2$ acts on $\mathbb{H}^2$ as a geometric reflection group with strict fundamental domain a regular hyperbolic $m$–gon with interior angles $\frac{2\pi}{n}$.
\end{prop}

This result extends to the three regular Euclidean tilings and the five regular spherical tilings, with analogous Coxeter structure for lower $m,n$. We can think of the Coxeter generators as the set of $m$ lines in hyperbolic space extending the edges of a regular $m$-gon with the given interior angles. 
In particular, an $\{m,n\}$ tiling can thus be thought of as a Coxeter system $(W_{m,n},S_m)$ where $S_m = \{s_1,...,s_m\}$, with an $m\times m$ Coxeter matrix $M$ defined by:
\[M_{i,j} =
\begin{cases}
    1 &\text{if } i = j\\
    2 &\text{if } |i -j| > 1\\
    \frac{n}{2} &\text{if } |i - j| = 1
\end{cases} \]

\subsection{Presheaves and Realizations}

In algebraic topology, spaces which are ``built by gluing together pieces of fixed shapes,'' much like planar tilings or more general triangulated manifolds, are modeled as presheaves on a category. The objects of this category correspond to the shapes of the decomposition and whose morphisms correspond to the relationships those shapes can have (typically one shape being a face of another, but also one shape projecting onto another or mapping isomorphically to itself). Most notably, (semi-)simplicial sets model spaces of arbitrarily high dimension by building them out of $k$-dimensional simplices for $k$ ranging over the natural numbers. Spaces are also often built as cubical sets, out of $k$-dimensional cubes rather than simplices. In each case, there is a \emph{geometric realization} functor from the ``combinatorial'' spaces of this form to the category of topological spaces.

All of the material in this section is standard background (see for instance \cite[Section III.7]{maclanemoerdijk}), but we include it for the reader's convenience and to help build the visual intuition we rely on later in the paper.

\begin{defi}
    A \emph{presheaf} on a category $\C$ is a functor $X \colon \C^{op} \to \Set$ from the opposite category of $\C$ to the category of sets and functions, where the set $X(c)$ is denoted $X_c$ for $c$ an object of $\C$. A \emph{map} or \emph{morphism} of presheaves is a natural transformation $X \to Y$, consisting of functions $X_c \to Y_c$ for each object $c$ in $\C$ which commute with the images of the morphisms of $\C$. The category of presheaves on $\C$ and morphisms of presheaves is denoted $\Ch$.
\end{defi}

\begin{ex}
    Consider the category $\G$ with two objects $0,1$ and two morphisms $s,t \colon 0 \to 1$ in addition to identities. A presheaf on $\G$ consists of two sets $X_0,X_1$ along with two functions $X_1 \to X_0$, which we also denote by $s,t$ for convenience. This is precisely the data of a directed graph, where $X_1$ is the set of edges, $X_0$ the set of vertices, and each edge $e \in X_1$ has $s(e)$ as its source and $t(e)$ as its target.
\end{ex}

One of the main ideas of this paper is that when a topological space is decomposed into pieces with fixed shapes in a suitably nice way, it can be ``modeled'' by a presheaf on a category whose objects correspond to those same shapes. To make this notion precise, we introduce the notion of \emph{geometric realization}, where a presheaf is used to construct a space or more generally an object in a category with colimits. To say a space is ``modeled'' by a presheaf is then to show that the space is isomorphic to the geometric realization of that presheaf.

In order to describe how an object in a category is "constructed" according to the data of a presheaf, we first describe how the presheaf itself can be constructed as a colimit in the same manner. This requires introducing both the pieces the presheaf will be constructed out of and the diagram of how those pieces fit together.

\begin{defi}
    For each object $c$ in $\C$, there is a \emph{representable} presheaf $y(c)$ given by $y(c)_{c'} = \Hom_\C(c',c)$ with the action on morphisms of $\C$ given by precomposition. The \emph{Yoneda embedding} is the functor $y \colon \C \to \Ch$ sending $c$ to $y(c)$ and acting on morphisms by postcomposition.
\end{defi}

\begin{ex}
    The representable graph $y(0)$ contains a single vertex and no edges, as there is only the identity morphism $0 \to 0$ and no morphisms $1 \to 0$. The other representable graph $y(1)$ contains two vertices (corresponding to the morphisms $s,t \colon 0 \to 1$) and one edge between them (corresponding to the identity morphism $1 \to 1$).
\end{ex}

\begin{defi}
    Given a presheaf $X$ on $\C$, its \emph{category of elements} $\smallint X$ has as objects the elements $x \in X_c$ for any $c$, a morphism $X_f(x) \to x$ for each morphism $f \colon c' \to c$ in $\C$ and $x \in X_c$, and identities and composites corresponding to those in $\C$.
\end{defi}

\begin{figure}
    \centering
    \includegraphics[width=0.4\linewidth]{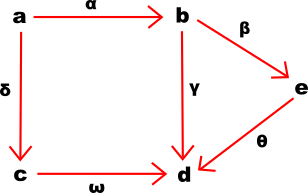}
    \caption{A picture of a graph with 5 vertices and 6 edges.}
    \label{fig:graph}
\end{figure}

\begin{ex}
    For the graph pictured in Figure~\ref{fig:graph}, its category of elements has objects $\{a,b,c,d,e,\alpha,\beta,\gamma,\delta,\omega\}$ and morphisms of the form $a \to \alpha \leftarrow b$ to each edge from its source and target vertices.
\end{ex}

The category of elements $\smallint X$ is perhaps most famous for its use in constructing the presheaf $X$ as a colimit of representables.

\begin{theorem}
    A presheaf $X$ on $\C$ is isomorphic to the colimit of the functor $\smallint X \to \C \xrightarrow{y} \Ch$ sending an object $x \in X_c$ to the representable presheaf $y(c)$.
\end{theorem}

\begin{ex}
    To construct a graph $X$ as a colimit in this manner, the functor $\smallint X \to \G \xrightarrow{y} \widehat{G}$ sends each vertex to the representable vertex $y(0)$ and each edge to the representable edge $y(1)$, with the morphisms corresponding to sources (resp. targets) sent to the source (resp. target) inclusion $y(0) \to y(1)$ of the vertex as the source (resp. target) of the edge. Thus the colimit can be constructed by beginning with the disjoint union of an edge for every edge in $X$ and a vertex for every vertex in $X$, then identifying each independent vertex with the source of every independent edge it is the source for in $X$, and likewise for targets. For the graph in Figure~\ref{fig:graph}, a portion of this colimit would look like identifying the source of $\beta$, the target of $\alpha$, and the vertex $b$ in the disjoint union picutred below.
    $$\bullet \xrightarrow{\quad\alpha\quad} \bullet \qquad\qquad \bullet_b \qquad\qquad \bullet \xrightarrow{\quad\beta\quad} \bullet$$
\end{ex}

As each presheaf is constructed as a colimit of representable presheaves which arise from only the objects in the category $\C$, we can construct objects of other categories based on a presheaf $X$ by specifying how the construction acts on each object of $\C$.

\begin{defi}
    For $\E$ a category with colimits and a functor $F \colon \C \to \E$, the \emph{geometric realization} functor $|-|_F \colon \Ch \to \E$ sends a presheaf $X$ to the colimit of the functor $\smallint X \to \C \xrightarrow{F} \E$ sending $x \in X_c$ to $F(c)$. When the functor $F$ is clear from context we will write simply $|X|$ for the geometric realization of a presheaf $X$.
\end{defi}

This construction is often used to build a topological space out of a presheaf, where the objects in a category of cell shapes are sent by the functor $F$ to topological spaces given by the envisioned shapes.

\begin{ex}
    There is a functor $\G \to \Top$ sending $0$ to the one-point space and $1$ to the line segment, with $s,t$ the inclusions of a point as either endpoint of the line segment. The corresponding geometric realization functor sends a graph $X$ to the space with points corresponding to the vertices in $X_0$ connected by line segments corresponding to the edges in $X_1$. For instance, the graph in Figure~\ref{fig:graph} geometrically realizes to a circle with a line through it: the circle is made up of line segments corresponding to the edges $\alpha,\beta,\theta,\omega,\delta$ while the edge $\gamma$ induces an additional line segment from the point $b$ to the point $d$.
\end{ex}

\section{Tilings as 2-Dimensional Presheaves}

\subsection{Categories for Directed $m$-gons}

    In order to describe regular tilings as presheaves, we first define categories whose objects correspond to the vertex, edge, and $m$-gon shapes that make up a tiling.

    \begin{defi}
        An \emph{$m$-gon category} for $m \ge 3$, which we denote by $\Dm$, is a category with objects and non-identity morphisms of the form
        \[
        \Dm \coloneq 
        \begin{tikzcd}[ampersand replacement=\&]
            0 \arrow[rr, bend right = 100, "v^1"] \arrow[rr, bend right = 100, draw = none, shift right = 3.5ex, "\vdots"] \arrow[rr, bend right = 100, shift right = 4ex, "v^m"'] \arrow[r, shift left = .75ex, "s"] \arrow[r, shift right = .75ex,"t"'] 
            \& 1 \arrow[r, shift left = 1.5ex, "d^1"] \arrow[r, draw = none, shift right = 2ex, "\vdots"] \arrow[r, shift right = 2.5ex, "d^{m}"'] 
            \& 2
        \end{tikzcd}
        \]
        such that $\{d^i \circ s, d^i \circ t\} = \{v_i,v_\ipo\}$ for $i=1,\ldots,m$.
    \end{defi}
        
    The objects of a category $\Dm$ represent a vertex, edge, and face respectively. The morphism $s$ corresponds to the source vertex of the edge and $t$ corresponds to the target vertex. The morphisms $d^i$ correspond to the $m$ edges of an $m$-gon, and the morphisms $v^i$ to the $m$ vertices of the $m$-gon. Therefore a presheaf $X$ on a category $\Dm$ consists of a set $X_0$ of vertices, a set $X_1$ of edges each equipped with a source and target vertex, and a set $X_2$ of $m$-gons equipped with $m$ edges whose source and target vertices agree in the manner expected of an $m$-gon. This is precisely the type of presheaf suitable to model $m$-gon tilings.

    There are $2^m$ different $m$-gon categories, as while the objects and morphisms are fixed their composition relations are not: the composite $d^i \circ s$ can be either $v^i$ or $v^\ipo$, where the other is $d^i \circ t$. This choice determines the ``direction'' of the edge $d^i$ in the $m$-gon, where if $d^i \circ s = v^i$ the edge points ``clockwise'' from $v^i$ to $v^\ipo$ and otherwise it points ``counterclockwise'' from $v^\ipo$ to $v^i$. 

    \begin{ex}\label{twopentagons}
        We write $y(2)$ for the representable functor $\Dm(-,2)$, whose components are as below.
        \[
        \begin{tikzcd}[ampersand replacement=\&]
            \{v^1,\ldots,v^m\}  \&
            \{d^1,\ldots,d^m\} \arrow[l, shift left = .75ex] \arrow[l, shift right = .75ex] \&
            * \arrow[l, draw=none, shift right = 3ex, "\vdots"]  \arrow[l, shift right = 2.5ex] \arrow[l, shift left = 2.5ex]
        \end{tikzcd}
        \]
    
    Below are pictures of $y(2)$ for two distinct pentagon categories, where the arrow on the $d^i$ edge points from the $d^i \circ s$ vertex to the $d^i \circ t$ vertex. We call an $m$-gon with this pattern of directed edges an \emph{$\Dm$-tile}.

    \begin{figure}[h]
        \centering
        \includegraphics[scale=.55]{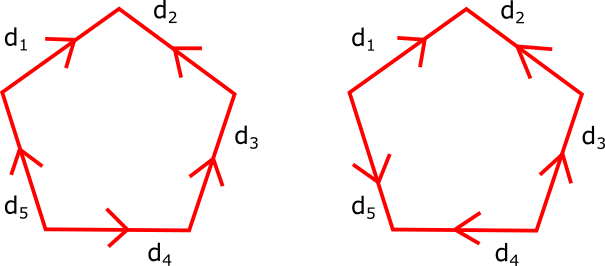}
        \caption{Visualizing the representable $y(2)$ for two pentagon categories}  
        \label{directedlabelledpentagons}
    \end{figure}

    The left pentagon category has composition defined as
    \[
    \begin{array}{rcccl}
        d^5 \circ t & = & d^1 \circ s & = & v^1 \\
        d^1 \circ t & = & d^2 \circ t & = & v^2 \\
        d^2 \circ s & = & d^3 \circ t & = & v^3 \\
        d^3 \circ s & = & d^4 \circ t & = & v^4 \\
        d^4 \circ s & = & d^5 \circ s & = & v^5, 
    \end{array}
    \]
    while in the right pentagon category it is defined as
    \[
    \begin{array}{rcccl}
         d^5 \circ s & = & d^1 \circ s & = & v^1 \\
         d^1 \circ t & = & d^2 \circ t & = & v^2 \\
         d^2 \circ s & = & d^3 \circ t & = & v^3 \\
         d^3 \circ s & = & d^4 \circ s & = & v^4 \\
         d^4 \circ t & = & d^5 \circ t & = & v^5.
    \end{array}     
    \]
    \end{ex}

    The simplest example of an $m$-gon category is the one in which every edge points in the clockwise direction.

    \begin{defi}
        The \emph{cyclic} $m$-gon category is given by
        \[
        \Dmc \coloneq 
        \begin{tikzcd}[ampersand replacement=\&]
            0 \arrow[rr, bend right = 100, "v^1"] \arrow[rr, bend right = 100, draw = none, shift right = 3.5ex, "\vdots"] \arrow[rr, bend right = 100, shift right = 4ex, "v^m"'] \arrow[r, shift left = .75ex, "s"] \arrow[r, shift right = .75ex,"t"'] 
            \& 1 \arrow[r, shift left = 1.5ex, "d^1"] \arrow[r, draw = none, shift right = 2ex, "\vdots"] \arrow[r, shift right = 2.5ex, "d^{m}"'] 
            \& 2
        \end{tikzcd}
        \]
        where $d^i \circ s = v^i$ and $d^i \circ t = v^\ipo$.

        \begin{figure}[h]
        \includegraphics[scale=.6]{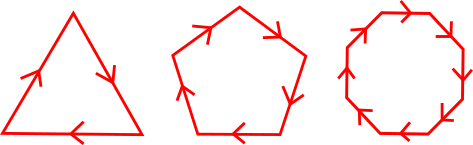}
        \caption{Visualizations of a triangle, pentagon, and octagon cyclic category.}
        \label{fig:cyclicexampletiling}
        \end{figure}
    
    \end{defi}

    As there are exactly $2^m$ distinct $m$-gon categories, we can index them by whether their edges point clockwise or counterclockwise.

    \begin{defi}\label{deltacm}
        Given an $m$-gon category $\Dm$, define the tuple
        \[
        \deldm \coloneq (\deldm_1,\ldots,\deldm_m) \in \{1,-1\}^m
        \]
        by letting $\deldm_i$ be 1 if $d^i \circ s = v^i$ in $\Dm$ and $-1$ otherwise where $d^i \circ t = v^i$. 
    \end{defi}

    In other words, $\deldm_i$ is 1 if the $d^i$ edge points clockwise and $-1$ if it points counterclockwise. 


    \begin{ex}
            \begin{figure}[h]
            \centering
            \includegraphics[scale=.5]{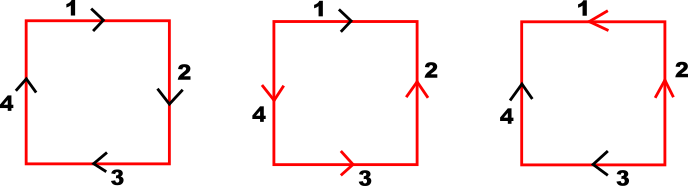}
            \caption{Three squares with (from left to right) edge reversal codes of $(1,1,1,1), (1,-1,-1,-1), (-1,-1,1,1)$. The black arrows are $1$s, and the red arrows are $-1$s, i.e. the arrows which differ from the cyclic category.}
            \label{edgereversalcodesquares}
            \end{figure}
    
        $\delta^{\Dmc} = (1,\ldots,1)$, while the tuples for the left and right pentagon categories in \cref{twopentagons} are respectively $(1,-1,-1,-1,1)$ and $(1,-1,-1,1,-1)$.
    \end{ex}

    We can also think of the tuple $\deldm$ as instructions for how to modify $\Dmc$ to get $\Dm$: each negative $\deldm_i$ corresponds to flipping the direction of the $d^i$ edge from clockwise to counterclockwise. 
    
    From this perspective, it makes sense to further consider tuples for instructions on how to modify any $m$-gon category to get another one.

    \begin{defi}
        For two $m$-gon categories, $\Dm$ and $\Dpm$, define the tuple
        \[
        \delddm \coloneq (\delddm_1,\ldots,\delddm_m) \in \{1,-1\}^m
        \]
        where $\delddm_i = \deldm_i \cdot \deldpm_i$.
    \end{defi}

    More concretely, we have that
    \[ 
        \delddm_i = 
        \begin{cases} 
          1 & d^i \underset{\Dm}{\circ} s = d^i \underset{\Dpm}{\circ} s \\
          -1 & d^i \underset{\Dm}{\circ} s = d^i \underset{\Dpm}{\circ} t \\ 
        \end{cases}
    \]
    where $\underset{\Dm}{\circ}$ denotes composition in $\Dm$ and $\underset{\Dm}{\circ}$ denotes composition in $\Dpm$.

    $\delddm$ is sometimes called an \emph{edge reversal code} as it describes how to reverse the edges of $\Dm$ to get $\Dpm$. Any such tuple $\delta$ conversely gives rise to a new $m$-gon category by reversing the edges of $\Dm$.

    \begin{ex}
        In all cases, $\delta^{\Dmc,\Dm} = \deldm$. Let $\C_4,\C'_4$ be the categories pictured below (left and right respectively), and we see $\delta^{\C_4,\C'_4} = (-1,1,-1,-1)$, where the red arrows represent the $-1$s, or the edges that have flipped from the first category. Note that $\delta^{\Dmc,\Dm}$ is relative - the color of the arrows shows the difference from the initial category, not necessarily the cyclic category.

        \begin{figure}[h]
        \centering
        \includegraphics[scale=.5]{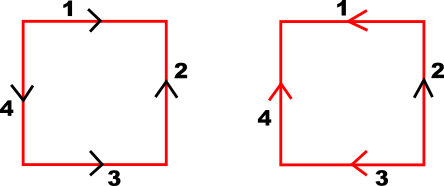}
        \caption{A visualization of $\delta^{\C_4,\C'_4}= (-1,1,-1,-1)$}
        \label{deltaCmCm'example}
        \end{figure}
        
    \end{ex}

\subsection{Enumerating Directed $m$-gons}

    In the undirected setting, the choice of $\{m,n\}$ completely determines a tiling. In the directed setting, in order for an $\{m,n\}$ tiling to be a presheaf on an $m$-gon category all of its face tiles need to ``look the same'' in terms of how their edges are labeled and directed, as in each face must be isomorphic to the representable $y(2)$.  
    
    Before we explore ways of achieving this, we first consider how many such possibilities for these tiles there are, as while there are exactly $2^m$ $m$-gon categories many of these are isomorphic to each other. We can describe these isomorphisms in terms of their effect on the representable $m$-gon, where they are generated by rotations and reflection.
        
    We can represent the categories $\Dm$ by the tuples $\deldm$. There is an action of the dihedral group $\Dim$ on $\{1,-1\}^m$, where if $\Dim$ is generated by a rotation $r$ and a reflection $f$. The action is given by
    \begin{equation}\label{action}
        r(\delta_1, \ldots, \delta_m) = (\delta_m, \delta_1, \ldots, \delta_{m-1})
    \end{equation}
    and 
    \[
    f(\delta_1, \delta_2, \ldots, \delta_m) = (-\delta_m, \ldots, -\delta_2, -\delta_1).
    \]

    This action translates to isomorphisms between $m$-gon categories, where writing $\Dm^\delta$ for the category corresponding to the tuple $\delta \in \{1,-1\}^m$ we have identity-on-objects isomorphisms $\rho \colon \Dm^\delta \cong \Dm^{r(\delta)}$ and $\psi \colon \Dm^\delta \cong \Dm^{f(\delta)}$ given by
    \[
        \rho(s) = s, \qquad \rho(t) = t, \qquad \rho(d^i) = d^\ipo
    \]
    and
    \[
        \psi(s) = t, \qquad \psi(t) = s, \qquad \psi(d^i) = d^{m+1-i}.
    \]

      \begin{figure}[h]
        \centering
        \includegraphics[scale=.4]{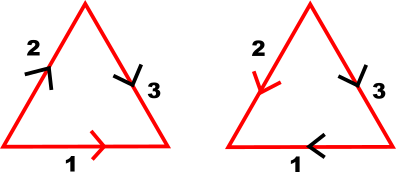}
        \caption{For the left triangle, $\deldm = (-1,1,1)$. Applying $r(\deldm)$, we get the right triangle, with edge reversal code $(1,-1,1)$.}  
        \label{trianglerotation}
    \end{figure}


    \begin{ex}

        Below are the 2 possible $\mathcal{D}_3$ categories, and the 4 possible $\mathcal{D}_4$ and $\mathcal{D}_5$ categories up to isomorphism - which correspond exactly to the number of possible directed $m$-gons, up to action by $\Dim$. The red arrows represent the -1s in the edge reversal code.
                    
    \begin{figure}[h]
        \centering
        \includegraphics[scale=.4]{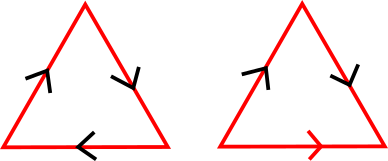}
        \caption{The only possibilites for directing a triangle. Notice that the non-cyclic directing gives a "canonical" directing of the edges, which makes it useful for constructions like simplicial homology.}
        \label{2triangles}
    \end{figure}

    \begin{figure}[h]
        \centering
        \includegraphics[scale=.4]{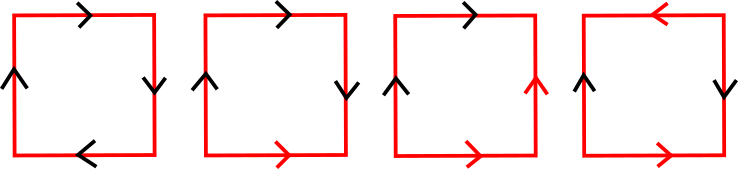}
        \caption{The 4 possible $\mathcal{C}_4$ categories.}
        \label{4squares}
    \end{figure}

    \begin{figure}[h]
        \centering
        \includegraphics[scale=.4]{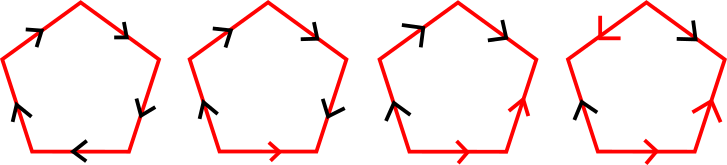}
        \caption{The 4 possible $\mathcal{C}_5$ categories. }
        \label{fig:C5categories}
    \end{figure}

    \end{ex}

    We now work out exactly how many types of directed $m$-gons there are for general $m$, up to isomorphism.

    \begin{prop}
        There are precisely 
    \[
        \frac{1}{2m}\left(
        2^m + 
        2\sum_{i=1}^{\frac{m-1}{2}} 2^{gcd(i,m)}
        \right)
    \]
        isomorphism classes of $m$-gon categories when $m$ is odd, and 
    \[
        \frac{1}{2m}\left(
        2^m + 
        \left(\frac{m}{2}+1\right)2^{\frac{m}{2}} +
        2\sum_{i=1}^{\frac{m}{2}-1} 2^{gcd(i,m)}
        \right)
    \]
        when $m$ is even.
    \end{prop}

    We arrive at this formula by counting the number of orbits of the action of $\Dim$ on $\{1,-1\}^m$.

    \begin{proof}
        Burnside's Orbit Counting Theorem states that the number of orbits is equal to 
    \[
        \frac{1}{|\Dim|}\sum_{g \in \Dim}|X^g|
    \]
    where $X^g$ is the subset of $X$ fixed by $g$. As the number of fixed points is constant within conjugacy classes, we need only compute $X^g$ for a $g$ representative of each conjugacy class, then multiply by the size of that class.

    Consider first the case where $m$ is even: we have conjugacy classes
    \[
        \left\{id_{D_m}\right\}, 
        \left\{r^i, r^{m-i}\right\}, 
        \left\{r^{m/2}\right\}, 
        \left\{f, r^2f, \ldots, r^{m-2}f\right\}, 
        \left\{rf, r^3f, \ldots, r^{m-1}f\right\}
    \]
    for $i \in \left\{1, \ldots, \frac{m}{2}-1 \right\}$. 
    Then we compute $|X^g|$ to get 
    \[
        \left|X^{id_{\Dim}}\right| = 2^m, 
        \left|X^{r^i}\right| = 2^{gcd(m,i)}, 
        \left|X^{r^{\frac{m}{2}}}\right| = 2^{\frac{m}{2}}, 
        \left|X^{f}\right| = 2^{\frac{m}{2}}, 
        \left|X^{rf}\right| = 0.
    \]
    Adding up these terms, the number of orbits is 
    \[
        \frac{1}{2m}\left(
        2^m + 
        \left(\frac{m}{2}+1\right)2^{\frac{m}{2}} +
        2\sum_{i=1}^{\frac{m}{2}-1} 2^{gcd(i,m)}
        \right).
    \]

    When $m$ is odd, we have conjugacy classes 
     \[
        \left\{id_{\Dim}\right\}, 
        \left\{r^i, r^{m-i}\right\}, 
        \left\{f, rf, r^2f, \ldots, r^{m-1}f\right\} 
    \]
    for $i \in \{1, \ldots, \frac{m}{2}-1 \}$. 
    We compute $X^g$ similarly as
    \[
        \left|X^{id_{D_m}}\right| = 2^m, 
        \left|X^{r^i}\right| = 2^{gcd(m,i)},  
        \left|X^{f}\right| = 0,
    \]
    and find that the number of orbits is 
    \[
        \frac{1}{2m}\left(
        2^m + 
        2\sum_{i=1}^{\frac{m-1}{2}} 2^{gcd(i,m)}
        \right).
    \]
    \end{proof}
    

\subsection{Tilings as Presheaves}

    The goal of this paper is to describe presheaves on $\Dm$ which correspond via geometric realization to an $\{m,n\}$ tiling. While a concrete description of the elements of such a presheaf is the subject of ongoing related work, for present purposes we will define such a presheaf $T$ by setting $T_0$, $T_1$, and $T_2$ to be ``the vertices, edges, and tiles respectively of an $\{m,n\}$ tiling,'' where the hard work goes into defining structure maps $s,t,d_1,...,d_m$ that make this a presheaf on some category $\Dm$ whose realization is in fact the tiled surface. 

    In this section, we define these presheaves in mostly combinatorial terms and prove the property that they geometrically realize to tilings of a surface.

    \begin{defi}
        There is a functor $\Dm \to \mathsf{Top}$ sending $0$ to the one-point space, $1$ to the interval, and $2$ to the standard $m$-gon (homeomorphic to the closed disk), where $s,t$ map to the source and target inclusions of the point into the interval and $d^1,\ldots,d^m$ to the $m$ successive edge inclusions of the interval into the $m$-gon. 

        We write $|X|$ for the geometric realization of a presheaf $X$ on $\Dm$ with respect to this functor. 
    \end{defi}

    This notion of geometric realization justifies calling elements of $X_0$ vertices, elements of $X_1$ edges, and elements of $X_2$ tiles for $X$ a presheaf on an $m$-gon category $\Dm$. The maps $d_i$ select the $i$-th edge of a tile, and the maps $s$ and $t$ select the source and target of an edge.

    \begin{defi}\label{directedtiling}
        A directed $\{m,n\}$ tiling on a space $X$ is a presheaf $T$ on $\Dm$, satisfying the following conditions:
        \begin{enumerate}
            \item $T$ is \emph{nonsingular}, meaning for each cell $x \in X_i$, the corresponding morphism $y(i) \to X$ is levelwise injective 
            \item Any two tiles $x_1,x_2 \in T_2$ have at most 1 edge in common
            \item Each vertex $v \in T_0$ has exactly $n$ edges $e$ such that either $s(e)=v$ or $t(e)=v$ 
            \item $|T|$ is homeomorphic to $X$
        \end{enumerate}
    \end{defi}

     The conditions in \cref{directedtiling} ensure that a presheaf can be visually represented as a regular tiling of the form $\{m,n\}$, as follows.
    \begin{enumerate}

        \item ensures that none of the edges or vertices of a tile in $T$ are identified with each other.

        \item ensures that no two distinct tiles of $T$ share more than one edge. This avoids situations such as a vertex appearing to be in the middle of an edge, or a pair of tiles forming a cylinder.

        \item ensures that each vertex of $T$ is adjacent to exactly $n$ edges (and accordingly $n$ tiles by (4) when $X$ is a surface), so that the tiling agrees with  the Schl\"afli symbol $\{m,n\}$.

        \item ensures that the tiling topologically models the space $X$.
    \end{enumerate}

\begin{defi}
        Let $T$ be a directed tiling on $\Dm$ and $x_1,x_2$ two tiles. We say $x_1$ and $x_2$ are adjacent if
        \[E = \bigcup_{i=1}^{m} d^i(x_1)\bigcap \bigcup_{j=1}^{m}d^j(x_2) \]
        contains exactly one edge. Observe, for $x_1 \not= x_2$, $|E| = 1$ or $0$. 
        A tile $x$ and an edge $e$ are said to be adjacent if $e=d_i(x)$ for some $i$. An edge $e$ and a vertex $v$ are said to be adjacent if $s(e)=v$ or $t(e)=v$.
    \end{defi}

The following is an immediate consequence of the condition that a tiling on a 2-dimensional surface has geometric realization homeomorphic to that surface, and that each tile includes into the presheaf injectively.

\begin{lemma}
    Given a directed tiling $T$ on a surface, for all $e \in  T_1$, $e$ is adjacent to exactly two tiles, i.e.$\left|\coprod\limits_{i=1}^m d_i^{-1}(e)\right| = 2$.
\end{lemma}



\begin{prop}
    For any directed $\{m,n\}$ tiling on a 2-dimensional surface without boundary, 
    the sets $T_0,T_1,T_2$ are in bijection with the sets of vertices, edges, and tiles in the undirected $\{m,n\}$ tiling, where the elements $d_1(x),...,d_m(x) \in T_1$ associated to $x \in T_2$ agree with the edges of the tile corresponding to $x$, and the vertices $s(e),t(e) \in T_0$ associated to $e \in T_1$ agree with the endpoints of the edge corresponding to $e$.
\end{prop}

While this applies to any surface, the main examples to keep in mind are the sphere, the Euclidean plane, and the hyperbolic plane.

\begin{proof}
Let $v$ be a vertex in $T_0$. By condition (3) in \cref{directedtiling}, we know that $v$ is adjacent to $n$ edges. By the previous lemma, each of these edges is adjacent to two faces. Then as $X$ is homeomorphic to a surface, there must be $n$ faces at each vertex, one between each pair of successive edges going around the vertex, so that $v$ belongs to an open disk in the geometric realization. Thus, $X$ behaves locally like an undirected $\{m,n\}$ tiling at each face, edge, and vertex. This completes the proof, as by condition (4) in \cref{directedtiling} the geometric realization of $T$ now gives a decomposition of the surface into $m$-gon tiles, edges, and vertices in the manner of an $\{m,n\}$ tiling.
\end{proof}
    


    \begin{remark}\label{conformity}
        Observe taking the undirected graph of the $\{m,n\}$ tiling and arbitrarily directing each of its edges does not generally form a directed tiling. This is due to the fact that these arbitrarily chosen directions may not correspond to a $\Dm$ presheaf, as different tiles may have edges directed in incompatible patterns. We will say ``a directed graph of the form $\{m,n\}$'' to refer to a graph whose underlying undirected graph is that of the $\{m,n\}$ tiling, and in the next section discuss conditions for this to be the underlying directed graph of a directed $\{m,n\}$ tiling. 
    \end{remark}
    
    For convenience, when $\{m,n\}$ are fixed we will sometimes call a directed tiling an \emph{alignment}, to emphasize that only the directions of the edges and their labels within each tile are being considered, rather than the undirected aspects of the tiling.


\section{Constructing Tilings by Edge Reversal}

    In general it is difficult to fully specify an alignment for a tiling of an infinite surface aside from a few simple examples, especially in the hyperbolic setting. One approach to simplifying the construction of alignments is to begin with a relatively straightforward directed $\{m,n\}$ tiling and describe how to modify it to get different alignments.
    
    
    In order to do this we introduce a presheaf on $\Dm$ that describes whether or not the direction of an edge in a tiling will be reversed when forming a new alignment. This allows us to model the process of arbitrarily reversing or preserving the direction of each edge. This in turn creates a new directed graph, from which we may be able to build a valid alignment. 

\subsection{The direction presheaf $\Lambda$}

    \begin{defi} 
        The direction presheaf on $\Dm$,
        \[
        \Lambda \coloneq 
        \begin{tikzcd}[ampersand replacement=\&]
                    \Lambda_0  \&
                    \Lambda_1 \arrow[l, shift left = .75ex] \arrow[l, shift right = .75ex] \&
                    \Lambda_2, \arrow[l, draw=none, shift right = 3ex, "\vdots"]  \arrow[l, shift right = 2.5ex, ""'] \arrow[l, shift left = 2.5ex, ""]
        \end{tikzcd}
        \]
        is given by
        \[
        \begin{tikzcd}[ampersand replacement=\&]
                    *  \&
                    \{1,-1\} \arrow[l, shift left = .75ex] \arrow[l, shift right = .75ex] \&
                    \{1,-1\}^m, \arrow[l, draw=none, shift right = 3ex, "\vdots"]  \arrow[l, shift right = 2.5ex, "\pi_1"'] \arrow[l, shift left = 2.5ex, "\pi_m"]
        \end{tikzcd}
        \]
        where $\pi_i \colon \{1,-1\}^m \to \{1,-1\}$ is the projection function onto the $i^{\textrm{th}}$ component of the product.
    \end{defi}
    
    An element of $\Lambda_2$ can be interpreted as a scheme or code for reversing the edges in the $\Dm$-tile. Accordingly, the tuple $\delddm \in \Lambda_2$ for some other $m$-gon category $\Dpm$ is precisely the code for producing the $\Dpm$-tile.

    \begin{remark}\label{groupstructure}
        Note that $\Lambda$ has the structure of an abelian group object among presheaves on $\Dm$ under multiplication. In other words, it lifts to a functor $\Dm^{op} \to \mathsf{Ab}$ as each of its components forms an abelian group under multiplication and the projection maps are homomorphisms. We will write $\cdot \; \colon \Lambda \times \Lambda \to \Lambda$ for its group operation as a map of presheaves.    
    \end{remark}
    
    The dihedral group $\Dim$ (generated by a rotation $r$ and reflection $f$) acts on the set $\Lambda_2 = \{1,-1\}^m$ in the manner 
    described in \eqref{action}:
    \[
    r(\delta_1, \ldots, \delta_m) = (\delta_2, \ldots, \delta_m, \delta_1)
    \]
    and 
    \[
    f(\delta_1, \delta_2, \ldots, \delta_m) = (-\delta_m, \ldots, -\delta_2, -\delta_1),
    \]
    corresponding to how the edge directions of a directed tile change under rotation and reflection.
    
    This action does not respect the group structure on $\Lambda_2$ (rotation does, reflection does not), but its orbits give rise to sub-presheaves of $\Lambda$ that will help us construct directed tilings.
    
    \begin{defi}
        Given two $m$-gon categories $\Dm$ and $\Dpm$, the $\Dpm$-restricted direction presheaf $\Lambda_{\Dpm}$ is a sub-presheaf of $\Lambda$ on $\Dm$ given by
        \[
            \Lambda_{\Dpm} \coloneq
            \begin{tikzcd}[ampersand replacement=\&]
                *  \&
                \{-1,1\} \arrow[l, shift left = .75ex] \arrow[l, shift right = .75ex] \&
                \deldm \cdot \left[ \deldpm \right]_{\Dim},  \arrow[l, draw=none, shift right = 3ex, "\vdots"]  \arrow[l, shift right = 2.5ex, "\pi_1"'] \arrow[l, shift left = 2.5ex, "\pi_m"]
            \end{tikzcd}
        \]
        where $\left[ \deldpm \right]_{\Dim}$ denotes the orbit of $\deldpm$ under the dihedral action on $\Lambda_2$.
    \end{defi}

    The set $\Lambda_{\Dpm}$ is the restriction of $\Lambda_2$ to only those edge reversal codes which turn a tile modeled by $\Dm$ into a tile modeled by $\Dpm$ up to rotation and reflection. This is expressed as the products of $\deldm$ with the orbit of $\deldpm$ under the dihedral action as applying such an edge reversal code to a tile modeled by $\Dm$ (whose edges are directed according to the code $\deldm$) cancels out the $\deldm$ leaving only the orbit of $\deldpm$. These are precisely the tiles modeled by $\Dpm$ up to rotation and reflection. 

    \begin{ex}
        Let $\C_5$ be the pentagon category with edge reversal code $(1,1,1,-1,-1)$, and $\C_5'$ the pentagon category with edge reversal code $(-1,1,1,1,1)$. The tiles corresponding to $\delta^{\C_5}$, $\delta^{\C_5'}$, and $r^{-1}(\delta^{\C_5'})$ are displayed in Figure~\ref{fig:enter-label}.
        
    \begin{figure}
        \centering
        \includegraphics[width=0.8\linewidth]{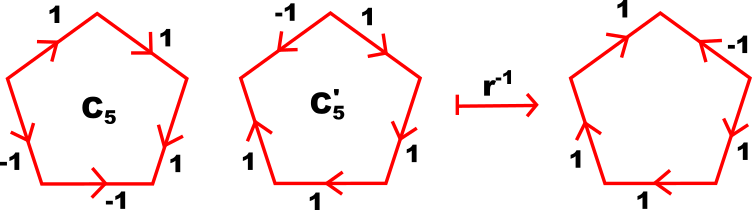}
        \caption{$\mathcal{C}_5$, $\mathcal{C}'_5$, and $r^{-1}$ applied to $\mathcal{C}_5'$.}
        \label{fig:enter-label}
    \end{figure}

        The elements of $(\Lambda_{\mathcal{C}'_5})_2$ corresponding to the dihedral group elements $e$ and $r$ can be computed as follows,
        $$\delta^{\C_5} \cdot \delta^{\C_5'} = (1,1,1,-1,-1)*(-1,1,1,1,1) = (-1,1,1,-1,-1)$$ 
        $$\delta^{\C_5} \cdot r^{-1}\left(\delta^{\C_5'}\right) = (1,1,1,-1,-1)*(1,-1,1,1,1) = (1,-1,1,-1,-1)$$ 
        and likewise for the remaining elements of $\mathsf{D}_5$.
    \end{ex}

\subsection{Edge reversal via maps of presheaves}

        
    A map from a presheaf on $\Dm$ into $\Lambda$ consists precisely of the information of whether or not to reverse the direction of each edge. We therefore call such a map an \emph{edge reversal}.

    \begin{defi}
        The underlying directed graph of a presheaf $X$ on $\Dm$ is given by the subdiagram $\begin{tikzcd}[ampersand replacement=\&] T_0 \& T_1 \arrow[l, shift left = .75ex, "t"] \arrow[l, shift right = .75ex, "s"'] \end{tikzcd}$, where $s$ and $t$ respectively determine the source and target of each edge.
    \end{defi}
        
    Suppose $T$ is an $\{m,n\}$ directed tiling on $\Dm$, and consider a map of presheaves $\tau \colon T \to \Lambda$. We can form a directed graph with vertices $T_0$ and edges $T_1$, but compared to the underlying graph of $T$ the direction of each edge $e$ is reversed if $\tau(e) = -1$ and preserved if $\tau(e) = 1$. From an undirected perspective this graph looks the same and still resembles an $\{m,n\}$ tiling. In fact, maps of this form allow us to describe any directed graph whose whose underlying undirected graph is that of the $\{m,n\}$ tiling. By naturality, a map $T \to \Lambda$ is completely determined by the function $T_1 \to \Lambda_1 = \{1,-1\}$.

    As discussed in \cref{conformity}, the tiles will not generally resemble the $\Dm$-tile, or even necessarily all look alike up to rotation and reflection. Thus, an edge-reversal does not necessarily describe a new alignment.

    \begin{ex}\label{5,4TtoLambdaEx}
    Let $T$ be the directed $\{5,4\}$ tiling on $\Dmc$ as seen in Figure~\ref{fig:5,4DirectedTiling},
    \begin{figure}[h]
        \includegraphics[scale=.6]{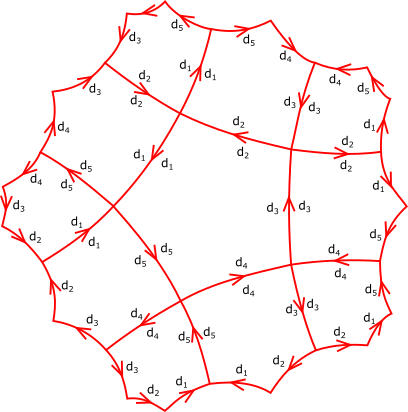}
        \caption{The $\{5,4\}$ directed tiling on $\mathcal{C}_5^\circlearrowright$.}
        \label{fig:5,4DirectedTiling}
    \end{figure}
    and let $\tau : T \to \Lambda$ send each element of $T_2$ to $(-1,1,1,1,-1)$, or equivalently send each edge labeled $d_1$ or $d_5$ to -1 and those labeled $d_2$, $d_3$, or $d_4$ to 1. $\tau$ thus carries the information of how to turn this tiling into the directed graph where the edges labeled $d_1$ or $d_5$ are reversed, as in Figure~\ref{fig:5,4tau(T)}. 
    
    \begin{figure}[h]
    \includegraphics[scale=.6]{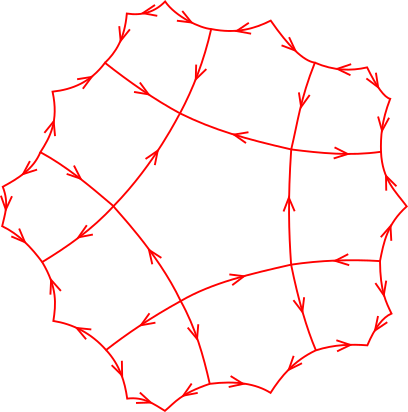}
    \caption{Visualization of $\tau(T)$.}
    \label{fig:5,4tau(T)}
    \end{figure}
    Note that it is no longer appropriate to include the edge labels on this graph as it does not necessarily extend to a presheaf on the same $m$-gon category. In fact in this case there exists no presheaf on $\Dmc$ that describes this directed graph, as the edges in the second graph do not point in a cyclic manner around the pentagon faces. However, this new graph does extend to a presheaf on a different pentagon category (see \cref{tauexample}). 
    \end{ex}

\subsection{Constructing the new alignment}

     A map $\tau \colon T \to \Lambda$ describes how to reverse the directions of some edges in a directed tiling $T$ to get a different directed graph. We now describe how to ensure that $\tau$ extends to an alignment. The key is to ensure that each $m$-gon face is given by the same directed $m$-gon up to rotation and reflection. For some other $m$-gon category $\Dpm$, by demanding that $\tau$ factors through $\Lambda_{\Dpm}$, we can recover a $\Dpm$ presheaf whose underlying directed graph is described by $\tau$.

    \begin{theorem}\label{newtiling}
        Given a directed tiling $T$ on an $m$-gon category $\Dm$ and a map $\tau \colon T \to \Lambda_{\Dpm}$ for another $m$-gon category $\Dpm$, there is a directed tiling $T'$ on $\Dpm$ whose underlying directed graph is obtained from that of $T$ by reversing the direction of every edge $e \in T_1$ for which $\tau(e) = -1$.
    \end{theorem}

    \begin{proof}
        Denoting the components and structure maps of $T$ as below,
        \[
        \begin{tikzcd}[ampersand replacement=\&]
            T_0  \&
            T_1 \arrow[l, shift left = .75ex, "t"] \arrow[l, shift right = .75ex, "s"'] \&
            T_2 \arrow[l, draw=none, shift right = 2.5ex, "\vdots", "d_1"']  \arrow[l, shift right = 2.5ex] \arrow[l, shift left = 2.5ex, "d_m"]
        \end{tikzcd}
        \]
        we define a presheaf $T'$ on $\Dpm$ as
        \[
        T' = \begin{tikzcd}[ampersand replacement=\&]
            T_0  \&
            T_1 \arrow[l, shift left = .75ex, "t'"] \arrow[l, shift right = .75ex, "s'"'] \&
            T_2, \arrow[l, draw=none, shift right = 2.5ex, "\vdots", "d_1'"']  \arrow[l, shift right = 2.5ex] \arrow[l, shift left = 2.5ex, "d_m'"]
        \end{tikzcd}
        \]
        where the sets of vertices, edges, and tiles are the same as those of $T$. 

        We now proceed to define the new structure maps, beginning with $s',t'$ which either agree with or swap $s,t$ on each edge $e$ individually according to $\tau(e)$.
        \[
        s'(e) \coloneq \left\{\begin{array}{cl}
            s(e) & \textrm{ if } \tau(e) = 1 \\
            t(e) & \textrm{ if } \tau(e) = -1
        \end{array}\right.
        \qquad\qquad
        t'(e) \coloneq \left\{\begin{array}{cl}
            t(e) & \textrm{ if } \tau(e) = 1 \\
            s(e) & \textrm{ if } \tau(e) = -1
        \end{array}\right.
        \]

        Now note that for a tile $x \in T_2$, if $\tau(x) = \delddm$ then we could define $d_i'(x) = d_i(x)$ for all $i$ and satisfy the equations governing how the maps $d'_i$ (restricted only to $x$) must compose with $s',t'$ to form a presheaf on $\Dpm$. This is because $\delddm$ encodes how each directed edge in a $\Dm$-tile must be preserved or reversed in order to produce a $\Dpm$-tile, and by naturality of $\tau$ the element $\tau(x) \in \Lambda_2$ encodes merely whether or not each of the $m$ edges of the tile $x$ are reversed.

        While $\tau(x)$ need not always be equal to $\delddm$, we have demanded that it be the product of $\deldm$ with an element of the orbit of $\deldpm$ under the action of the dihedral group on $\Lambda_2$ from \eqref{action}. Therefore for each $x \in T_2$, there is a group element $\sigma_x \in \Dim$ which $\tau(x) = \deldm \cdot \sigma_x(\deldpm)$.\footnote{This action is only faithful on certain orbits, so there may be many such choices, but considering more careful methods for making those choices is beyond the present scope of this work.}

        We now use the standard action of $\Dim$ on the set $\{1,\ldots,m\}$, in which $r(i) = \ipo$ and $f(i) = m + 1 - i$, in order to define the face maps of $T'$ as
        \[
            d_i'(x) = d_{\sigma_x(i)}(x).
        \]
        In other words, the labels $d_1(x),\ldots,d_m(x)$ on the $m$ edges of the tile $x$ in $T$ are rotated and/or reflected in precisely the manner needed to make the tile $x$ in $T'$, with its edges reversed or preserved according to $\tau$, resemble a directed $m$-gon of the sort modeled by $\Dpm$. 

        More precisely, we have by definition of $d'_i$ and $\sigma_x$ that $\tau(d_i'(x)) = \delddm_i$ for all $x \in T_2$, so that either $s'(d_i'(x))$ or $t'(d_i'(x))$ is $v_i'(x)$ according to $\deldpm_i = \delddm_i \cdot \deldm_i$. Hence $T'$ forms a presheaf on $\Dpm$, which is evidently a directed tiling as its sets of vertices, edges, and tiles are the same as $T$ and the edges of each tile and vertices of each edge are merely permuted compared to those of $T$.
    \end{proof}

    \begin{ex}\label{tauexample}
        Recall example~\ref{5,4TtoLambdaEx}, which describes a map from a presheaf $T$ on $\mathcal{C}_5^\circlearrowright$ to $\Lambda$. Let this map be $\tau$, and let $\mathcal{C}_5$ have relations visualized as in Figure~\ref{Fig:Ex3.9C5Cat}:
        
        \begin{figure}[h]
            \centering
            \includegraphics[scale=.25]{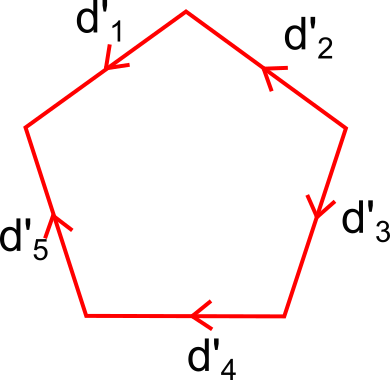}
            \caption{Visualization of a $\mathcal{C}_5$ category.}
            \label{Fig:Ex3.9C5Cat}
        \end{figure}
        
        Then, $\delta^{\mathcal{C}_5^\circlearrowright,\mathcal{C}_5} = (1,1,-1,-1,-1)$. Observe for $x\in T_2$, $\tau(x) \in [\delta^{\mathcal{C}_5^\circlearrowright,\mathcal{C}_5}]_{\Dim}$, then from Theorem~\ref{newtiling}, we are able to construct the presheaf, $T'$, that agrees with $\tau$.
        First observe for any $x \in T_2$, $\tau(x) = (-1,1,1,1,-1)$. Therefore, for $\sigma = rf \in \Dim$, $\sigma \tau(x) = \delta^{\mathcal{C}_5^\circlearrowright,\mathcal{C}_5}$. Then as Theorem~\ref{newtiling} suggests, define $d_i'(x) = d_{\sigma_x(i)}(x)$. Thus we have:
        \[
            d_1' = d_1,\
            d_2' = d_5,\
            d_3' = d_4,\
            d_4' = d_3,\
            d_5' = d_2
        \]
        Since $\tau$ only reverses $d_1$ and $d_5$, then $d_1'$ and $d_2'$ have swapped target and source relations. Thus, we have the following source and target relations:
        \[
            t(d_1') = t(d_2'),\
            s(d_2') = s(d_3'),\
            t(d_3') = s(d_4'),\
            t(d_4') = s(d_5'),\
            t(d_5') = s(d_1')
        \]

    Observe with these new relations and maps we get the following presheaf on $\mathcal{C}_5$ whose elements are the same as those of $T$. As $\tau$ does not depend on choice of tile, there is a unique way to label the edges of each tile.

    \begin{figure}[h]
        \centering
        \includegraphics[scale=.6]{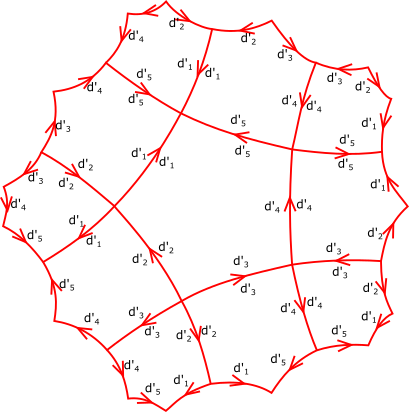}
        \caption{Geometric realization of tiling recovered from $\tau$ described in Example~\ref{5,4TtoLambdaEx}.}
    \end{figure}
    \end{ex}

\section{Reflection Generated Tilings}

    We now have a way to obtain a new alignment from a given one by changing the direction of certain edges specified by a map $\tau \colon T \to \Lambda_{\Dpm}$ for some $\Dpm$. However, defining $\tau$ for an infinite tiling requires an infinite amount of data, namely a choice of $\pm 1$ for each edge in $T$. In this section we introduce a class of edge reversals on the plane which can be defined based on only finitely many choices, and as such provide a wide range of examples for directed tilings themselves.

    We will assume henceforth that we are working only with tilings of the Euclidean or hyperbolic planes. Many of these results and techniques will apply similarly to more general surface tilings, but this is beyond the present scope.

\subsection{The Reflective Tiling}

    We begin by discussing a relatively straightforward alignment which will serve as a convenient ``base tiling'' for constructions of new alignments using \cref{newtiling}. 

    
    \begin{defi}\label{Reflective Tiling}
        A directed tiling $T$ is reflective if for each edge $e$ between adjacent tiles $x_1,x_2$, $e = d_i(x_1) = d_i(x_2)$ for some $i \in \{1,\ldots,m\}$. 
    \end{defi}

    In other words, each edge has the same label in both of the two tiles it lies between. This property in fact implies that all of the labels in those two tiles agree up to ``reflection'' across their shared edge, as we demonstrate with an example.

    \begin{ex}
        For the cyclic square category $\mathcal{C}_4^\circlearrowright$, a reflective $\{4,4\}$ tiling looks like that in Figure~\ref{fig:44reflective}, where the number $i$ on each edge indicates that edge is $d_i$ of both tiles it lies between. Note that the labels for each pair of adjacent tiles are reflected across their shared edge, and around each vertex the edge labels alternate between $i$ and ${i+1\;(\mathrm{mod}\;4)}$ for some $i$ between $1$ and $4$.

        \begin{figure}[h]
            \centering
            \includegraphics[scale=.35]{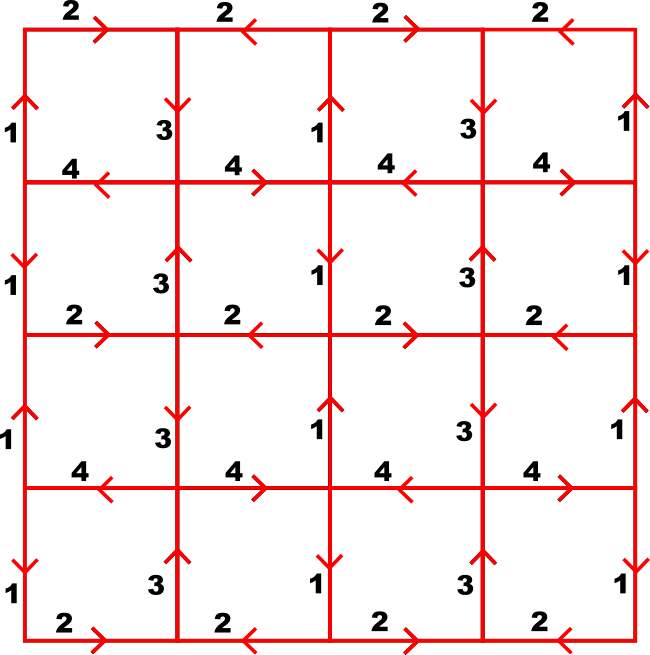}
             \caption{A 4x4 piece of the reflective $\{4,4\}$ tiling with cyclic tiles.}
             \label{fig:44reflective}
        \end{figure}
    \end{ex}

    \begin{lemma}
        If an $\{m,n\}$ directed tiling on $\Dm$ is reflective, then $n$ is even.
    \end{lemma}

    \begin{proof}
        Fix a vertex $v$ in the tiling and choose an edge $e_1$ with that vertex as an endpoint. We will assume that $v = s(e_1)$, though if instead $v=t(e_1)$ the exact same argument will apply, and that $e = d_i(x_1) = d_i(x_2)$ for $x_1,x_2$ the two tiles which share $e$ as an edge. 
        
        In each tile $x$ of a presheaf on $\Dm$, the source vertex of $d_i(x)$ is also the source or target of either $d_{\ipo}(x)$ or $d_{i-1\;(\mathrm{mod}\;m)}(x)$; we will assume it is the source of $d_\ipo(x)$ but similarly the same argument applies regardless.

        Therefore, if $v = s(e_1)$ then (under our arbitrary assumptions) we also have $v = s(d_\ipo(x_2))$; we will denote $d_\ipo(x_2)$ by $e_2$. Likewise, if $x_3$ is the second tile with $d_\ipo(x_3) = e_2$ then $s(d_i(x_3)) = v$; we denote $d_i(x_3)$ by $e_3$. Proceeding in this manner we get a sequence of edges $e_1,e_2,e_3,\ldots$ with source $v$ and labels alternating between $i$ and $\ipo$. As there are exactly $n$ edges with source $v$, the alternating pattern of labels means that $n$ must be even as otherwise the edge labels could not continually alternate around $v$.
    \end{proof}

    The following notion of a track will provide the basis for many induction arguments throughout this section.
    
    \begin{defi}
        Given two tiles $x,y \in T_2$ for a reflective tiling $T$, a track from $x$ to $y$ is a finite sequence of tiles 
        $$x=x_0,x_1,...,x_k=y$$
        such that $d_{j_i}(x_i)=d_{j_i}(x_{i-1})$ for $i=1,...,k$ and $j_i \in \{1,...,m\}$ for each $i$. The sequence $j_1,...,j_k$ is called the route of the track and $k$ is called the length.
    \end{defi}

    \begin{prop}\label{reflectiveexists}
        For any $m$-gon category $\Dm$ and any even number $n \ge 4$, there is a reflective directed $\{m,n\}$ tiling on $\Dm$.
    \end{prop}

    In the cyclic case, the reflective tiling is defined to have the $n$ edges at each vertex alternate between incoming and outgoing edges, with each edge having the same label $d_i$ for both of the tiles that share it for some $i$. For more general directed $m$-gons, the edge directions and labels are similarly determined by reflection across each edge as illustrated in the picture below.

    \begin{figure}[h]
        \centering
        \includegraphics[scale=.6]{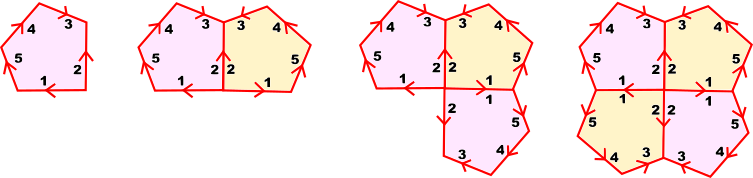}
        \caption{First, we reflect over the edge labeled 2, then over the edge of the new tile labeled 1. This third tile has been reflected twice, so is now the same as the original tile. Reflecting for a third time, the tile must line up with the first tile, because of the odd number of reflections. Here, $n=4$, but this strategy works for all even $n$, and for all starting $\Dm$ categories. We can continue this process around any of the new vertices, and it will stay consistent, thus generating an infinite reflective tiling.}
    \end{figure}

    While we include a technical proof of this claim, we warn the reader that it does not provide any intuition beyond this informal and visual description.

    \begin{proof}
        First consider the cyclic $m$-gon category $\Dmc$ and let $n \geq 4$ be even.

        Consider $(W_{m,n},S_m)$ the Coxeter system representing the reflections of the $\{m,n\}$ tiling as introduced in \cref{mncoxeter}, and choose a ``base tile'' $x_0$ in the undirected $\{m,n\}$ tiling. The generators $s_1,...,s_m$ of the Coxeter group $W_{m,n}$ can be represented as the reflections across the geodesics going through the edges $d_1(x_0),...,d_m(x_0)$ respectively. Each element $w$ of the Coxeter group acts on the undirected tiling and sends $x_0$ to a tile $w(x_0)$, and as this action on tiles is free there is a unique element $w$ for each tile $x$. 
        
        Observe that there is a group homomorphism, $\epsilon \colon W_{m,n} \to \{1,-1\}$ mapping each $s_i$ to $-1$. This induces a \{1,-1\}-coloring of our tiles where each tile has the opposite color to all of those adjacent to it. We can now set the edge directions by demanding that in the tiles colored by a 1 the edges point clockwise around the tile, and in the tiles colored by a -1 the edges point counterclockwise. This is consistent as if an edge points clockwise around one tile it points counterclockwise around the other tile adjacent to it and vice versa.

        It remains now only to determine the labels $d_1,...,d_m$ of the edges on each tile. Consider a tile $x$ and let $w$ be the element of $W_{m,n}$ with $x = w(x_0)$; $w$ can be written as a product of generators $s_{i_1} \cdots s_{i_k}$, though not necessarily uniquely. We will proceed by induction on $k$.
        
        As a base case, choose any clockwise labeling of the edges of $x_0$ by $d_1,...,d_m$. Now for induction assume that every tile with a track of length $k$ or less from $x_0$ has its edges labeled such that each edge has the same label in both adjacent tiles, as required by the definition of a reflective tiling. Let $x$ be a tile with corresponding Coxeter group element $w$ where $w = s_{i_1} \cdots s_{i_{k+1}}$. If $x$ has a track of length $k$ or less from $x_0$, then its edges are already labeled by assumption. Otherwise, there is a track from $x_0$ to $x$ with route $i_1,...,i_{k+1}$, where the route is determined by the existing edge labels and $x$ must be adjacent to the tile $s_{i_1} \cdots s_{i_k}(x_0)$ along its edge labeled $d_{i_{k+1}}$ (as by \cref{mncoxeter} $x$ is reached from $x_0$ by successive reflections along these edges). We then must label this edge as $d_{i_{k+1}}$ in the tile $x$, which based on the coloring of $x$ determines the remaining edge labels in $x$ in a clockwise or counterclockwise manner. That this construction does not depend on the choice of factorization of $w$ is a consequence of the Coxeter relations.

        Now consider any $m$-gon category $\Dm$, and let $T$ be the reflective tiling on $\Dmc$ defined above. There is a map $\tau \colon T \to \Lambda$ of presheaves on $\Dmc$ given on tiles by $\tau(x) = \delta^{\Dm}$, where $\delta^{\Dm}$ (\cref{deltacm}) denotes how to reverse edges in an $m$-gon with clockwise cyclic edge directions to get an $m$-gon with edge directions determined by the category $\Dm$. This is well defined by the reflective property of edge labels in $T$, and by definition $\tau$ factors through $\Lambda_{\Dm}$. By \cref{newtiling}, we get a directed $\{m,n\}$ tiling on $\Dm$ with the same reflective edge labels in each vertex but with the necessary edges reversed to make each tile of the form given by $\Dm$.
    \end{proof}

    Note that while this construction is not unique, as it relied on a choice of base tile and a choice of how to label the edges of that tile, it is unique up to isomorphism. This is because the different edge labelings of the base tile agree up to rotation and any other choice of base tile would at worst require a reflection to swap the tiles colored with a 1 and those colored with a -1.

\subsection{Reversal-Closed Transformations}

Our goal in this section is to generate new directed tilings by reversing edges in the reflective tiling in a recursive manner. In essence, this involves identifying when applying an edge-reversal code $\delta$ to the edges of a $\Dm$-tile results in a directed $m$-gon isomorphic to a $\Dm$-tile. In other words, when is it the case that $\deldm \cdot \delta = \sigma(\deldm)$ for some $\sigma \in \Dim$?

This question is immediately answered by the set $\{\deldm \cdot \sigma(\deldm) | \sigma \in \Dim\} \subset \{1,-1\}^m$, but crucially this subset is not preserved by replacing $\deldm$ with $\sigma(\deldm)$. This raises a second, more challenging question: when is it the case that $$\deldm \cdot \sigma_1(\deldm) \cdot \sigma_2(\deldm)$$ agrees with $\sigma_3(\deldm)$ for some $\sigma_3 \in \Dim$? This scenario represents when an edge-reversal code $\deldm \cdot \sigma_1(\deldm)$ can be applied to a directed $m$-gon with edge directions $\sigma_2(\deldm)$ in the orbit of $\deldm$ under the dihedral action, and still remain in that orbit.

The purpose of this question is to identify sets of edge reversal codes which can be applied to $\deldm$ repeatedly while always producing $\Dm$-tiles (as in, never leaving the orbit of $\deldm$). This motivates the following definition.

\begin{defi}
    A subset $\Gamma$ of the dihedral group $\Dim$ is \emph{reversal-closed} for $\delta \in \{1,-1\}^m$ if for all $\sigma_1,\sigma_2 \in \Gamma$, there exists at least one $\sigma_3 \in \Gamma$ such that $$\delta \cdot \sigma_1(\delta) \cdot \sigma_2(\delta) = \sigma_3(\delta).$$

    We write $\Lambda^\Gamma_2$ for the set $\{\delta \cdot \sigma(\delta) | \sigma \in \Gamma\}$ of edge reversal codes given by limiting the orbit of $\delta$ under the action of $\Dim$ to the subset $\Gamma$.
\end{defi}

As every element of the group $\{1,-1\}^m$ has order 2, this condition is equivalent to demanding that $\sigma_1(\delta) \cdot \sigma_2(\delta) \cdot \sigma_3(\delta) = \delta$ with the same quantifiers for $\sigma_1,\sigma_2,\sigma_3$.

\begin{ex}
    For $\delta = (1,1,1,-1,-1)$, the subset $\Gamma = \{e,f,rrr,frrr\}$ is reversal-closed as
    \[
    e(\delta) \cdot \sigma(\delta) \cdot \sigma(\delta) = \delta 
    \qquad\textrm{and}\qquad f(\delta) \cdot rrr(\delta) \cdot frrr(\delta) = \delta,
    \]
    where the first equation holds for any $\sigma \in \Gamma$. These two equations (coupled with commutativity of multiplication in $\{1,-1\}^m$) suffice to show that $\Gamma$ is reversal-closed.

    In this case $\Gamma$ is maximal, but for any $\delta$ if we have the equation $\sigma_1(\delta) \cdot \sigma_2(\delta) \cdot \sigma_3(\delta) = \delta$, then the set $\{e,\sigma_1,\sigma_2,\sigma_3\}$ is reversal-closed. It is not necessarily maximal however, and building reversal-closed subsets based on multiple such equations is more challenging as adding additional elements may introduce conflicts.
\end{ex}

The set $\Lambda^\Gamma_2$ is defined so that the edge-reversal codes it contains can be applied to $\delta$ arbitrarily many times without leaving the orbit of $\delta$ under $\Dim$, or moreover under the subset $\Gamma$.

\begin{lemma}\label{gammareuse}
    If $\Gamma$ is a reversal-closed subset for $\delta$, then for any $\tau \in \Lambda^\Gamma_2$ and $\sigma \in \Gamma$ there exists $\sigma' \in \Gamma$ such that $\sigma(\delta) \cdot \tau = \sigma'(\delta)$.
\end{lemma}

\begin{proof}
    Directly from the definitions, $\tau = \delta \cdot \sigma''(\delta)$ for some $\sigma'' \in \Gamma$, and $\sigma(\delta) \cdot \delta \cdot \sigma''(\delta)$ agrees with $\sigma'(\delta)$ for some $\sigma' \in \Gamma$.
\end{proof}

We describe additional properties and examples of reversal-closed subsets in \cref{appendixreversal}.

In order to use a reversal-closed subset of permutations for $\deldpm$ and the associated edge-reversal codes to determine how to reverse edges of a reflective tiling in a systematic way, we will make use of a restriction of the presheaf $\Lambda$ based on a reversal-closed subset.

\begin{defi}
    For $m$-gon categories $\Dm,\Dpm$ and a reversal closed subset $\Gamma$ for $\deldpm$, we define the presheaf $\Lambda^\Gamma$ by the diagram 
    \[
        \Lambda^\Gamma \coloneq
        \begin{tikzcd}[ampersand replacement=\&]
            *  \&
            \{-1,1\} \arrow[l, shift left = .75ex] \arrow[l, shift right = .75ex] \&
            \left\{\deldpm \cdot \sigma(\deldpm) \middle| \sigma \in \Gamma\right\}.  \arrow[l, draw=none, shift right = 3ex, "\vdots"]  \arrow[l, shift right = 2.5ex, "\pi_1"'] \arrow[l, shift left = 2.5ex, "\pi_m"]
        \end{tikzcd}
    \]
\end{defi}

\subsection{Local Reflection Schema}


We've discussed how to globally change one directed tiling into another, but now we illustrate how to describe this change more locally, from one tile to its adjacent tiles. We fix throughout an $m$-gon category $\Dm$.

    \begin{defi}
        $\P$ will denote the presheaf on $\Dm$ consisting of a single central tile and all of its adjacent tiles, with edges labeled as if it were in a reflective tiling on $\Dm$.

        More concretely, $\P_2 = \{s_0,s_1,\ldots,s_m\}$ with $d_i(s_i) = d_i(s_0)$ for $i=1,\ldots,m$.
    \end{defi}

    \begin{ex}
        Below left is a picture of $\P$ over the cyclic square category $\mathcal{C}_4^\circlearrowright$, and below right is a picture of $\P$ over the cyclic pentagon category.
        \begin{figure}[h]
            \centering
            \includegraphics[height=4cm]{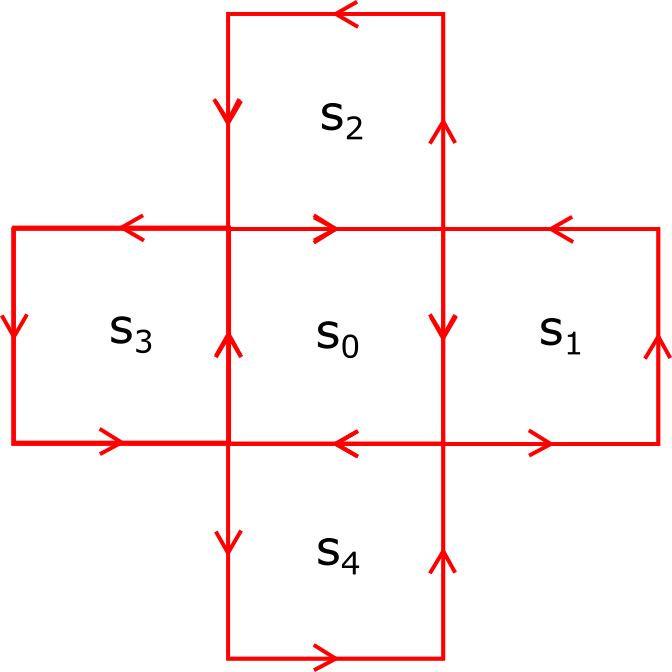}
            \qquad\qquad
            \includegraphics[height=4cm]{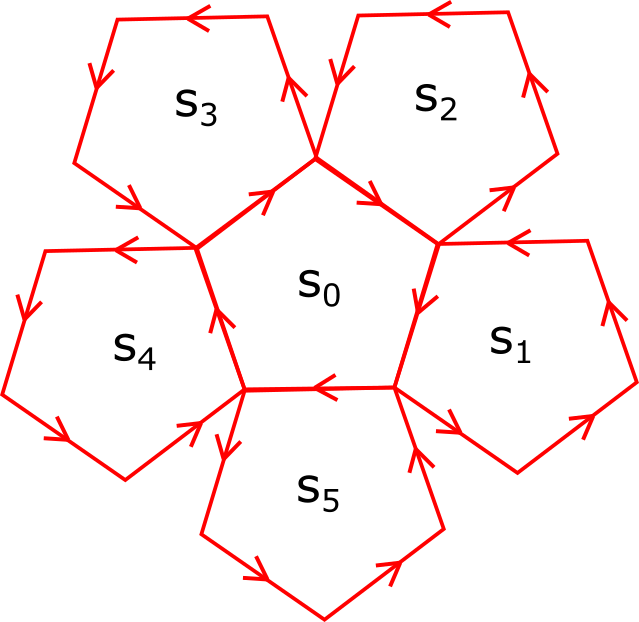}
            \caption{Examples of $\P$ for a square and pentagon category, labeled with $\{s_i\}$}
        \end{figure}
    \end{ex}
    
    Based on the definition of $\P$ we can observe its most important property: for $T$ a reflective tiling on $\Dm$ and $x \in T_2$ any tile, there is a unique morphism of presheaves $\bar x \colon \P \to T$ with $\bar x(s_0) = x$.

    \begin{figure}[h]
        \centering
        \includegraphics[width=0.5\linewidth]{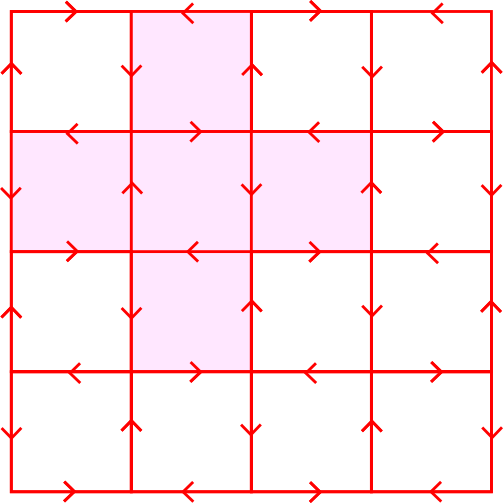}
        \caption{Depiction of $\bar x \colon \P \to T$ with $\bar x(s_0) = x$, with $x$ the center pink tile}
        \label{fig:xbarontiling}
    \end{figure}


    We can also extend maps to $\Lambda$ from the center tile $s_0$ to all of $\P$ by ``reflection.''

    \begin{defi}
        For $f \colon y(2) \to \Lambda$, the morphism $\bar f \colon \P \to \Lambda$ is defined by $\bar f(s_i) = f$ for $i=0,1,\ldots,m$.
    \end{defi}

    $\bar f$ can be regarded as a canonical extension of $f$ to $\P$ along the morphism $s_0 \colon y(2) \to \P$. 

    
    We wish to use $\P$ to describe a choice of how to reverse the edges in the tiles adjacent to a specified tile in a directed tiling, relative to both the reflective tiling and a separate choice of edge-reversals for the center tile,\footnote{Separating the edge-reversal of the center tile from its relationship with the edge-reversals of its adjacent tiles is not technically necessary, but vastly simplifies the definition of both a reflection scheme and the resulting reflection-generated tilings.} in such a way that extends the directed $m$-gon type of the center tile to the adjacent tiles. 

    \begin{defi}\label{reflectionscheme}
        For $m$-gon categories $\Dm,\Dpm$, a reversal-closed subset $\Gamma$ for $\deldpm$, a \emph{$\Gamma$-reflection scheme} over $\Dm$ is a morphism of presheaves $\phi \colon \P \to \Lambda^\Gamma$ such that $\phi(s_0) = (1,\ldots,1)$.
    \end{defi}

        
    Requiring that $\phi$ lands in $\Lambda^\Gamma$ will ensure that the surrounding tiles preserve the $m$-gon type of the center tile when we use $\phi$ to construct directed tilings. The condition on $\phi(s_0)$ is for convenience to make $\phi$ describe a choice of edge-reversals of the surrounding tiles relative to a choice of edge-reversal of the center tile. 

    \begin{ex}
    In Figure~\ref{fig:phipentagonexample}, we display an example of a reflection scheme for the pentagon category $\Dm$ with $\deldm = (-1,-1,-1,1,1)$. Here, $\phi(s_0)=(1,1,1,1,1)$,  $\phi(s_1)=(1,1,-1,-1,-1)$, $\phi(s_2)=(1,1,1,1,1)$, $\phi(s_3) = \phi(s_4) = \phi(s_5)=(-1,-1,1,1,1)$.   The black arrows indicate the ones changed by the $\phi_i's$. We can also observe the dihedral actions that each $\phi$ map represents: $\phi_1$ is $r$, $\phi_2$ is $e$, $\phi_3, \phi_4,\phi_5$ are $r^2$. We can check that these form a reversal closed subset.
     \begin{figure}[h]
        \centering
        \includegraphics[width=0.9\linewidth]{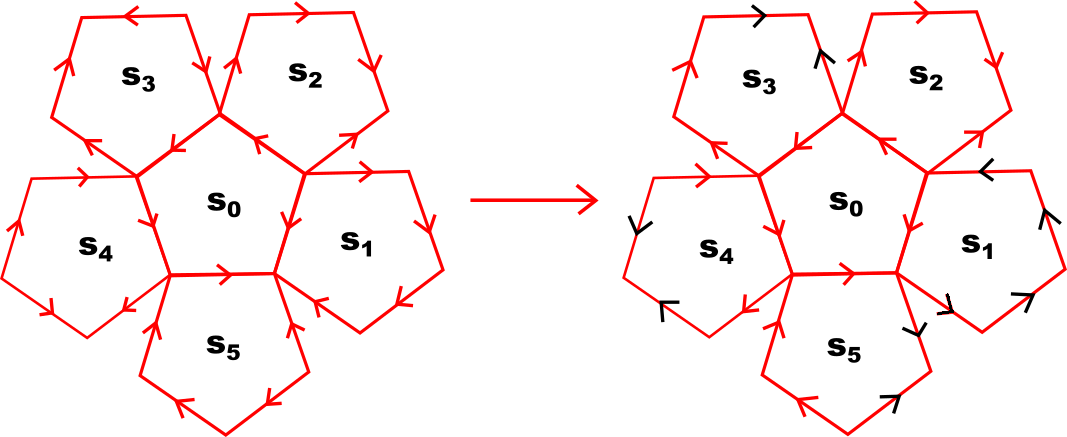}
        \caption{An example of $\phi$ on a $\mathcal{C}_5$ category.}
        \label{fig:phipentagonexample}
    \end{figure}
  
    \end{ex}


    Recall from \cref{groupstructure} that $\Lambda_2$ forms an abelian group under multiplication (namely $(\mathbb{Z}/2)^m$) in which every element has order 2. The following properties are immediate consequences of this. 

    \begin{lemma}\label{basicproperties}
       For all $i,j$, $\phi(s_i)^2= (1, \ldots , 1)$, and $\phi(s_i) \cdot \phi(s_j) = \phi(s_j) \cdot \phi(s_i)$.
    \end{lemma}

    \begin{lemma}\label{coxeterhomomorphism}
        A reflection scheme $\phi$ can be equivalently interpreted as a group homomorphism $W_{m,n} \to \Lambda_2$ from the Coxeter group for any $n$ divisible by 4, where $\phi(s_i)$ suggestively denotes the image of the generator $s_i$ of $W_{m,n}$.
    \end{lemma}

    \begin{proof}
        To show that this assignment on the generators $s_i$ extends to a homomorphism, it suffices to show that $$\left(\phi(s_i)\cdot\phi(s_{\ipo})\right)^{n/2} = (1,\ldots,1)$$
        for all $i=1,...,m$. But as $n$ is presumed to be divisible by 4, this will always be true as $n/2$ is even.
    \end{proof}

    To see further why it is important for $n$ to be divisible by 4 as opposed to simply even, we show that in the case of $n = 2$ (mod 4) this comparison would significantly constrain the reflection scheme.

    \begin{lemma}
        Any homomorphism $\phi \colon W_{m,n} \to \Lambda_2$ for $n = 2$ (mod 4), such that the $i$-th component of $\phi(s_i)$ is 1 for all $i$, is constant at the identity. 
    \end{lemma}

    \begin{proof}
        By definition, $\left(\phi(s_i) \cdot \phi(s_{i+1})\right)^{n/2}= (1, \ldots , 1)$. By \cref{basicproperties}, we can rearrange this equation to get 
        \[
            \phi(s_i)^{\frac{n}{2}-1} \cdot \phi(s_{i+1})^{\frac{n}{2}-1} \cdot \phi(s_i) \cdot \phi(s_{i+1}) = (1, \ldots , 1).
        \]
        As $\frac{n}{2}-1$ is even, by the previous lemma this equation reduces to 
        \[
            \phi(s_i) \cdot \phi(s_{i+1}) = (1, \ldots , 1).
        \]
        Again as each $\phi(s_i)$ has order 2 in $\Lambda$, this shows that $\phi(s_i) = \phi(s_{i+1})$ and hence that $\phi(s_i) = \phi(s_j)$ for all $i,j \in \{1, \ldots , m\}$. By definition, the $i$-th component of $\phi(s_i)$ is 1, so we can conclude that $\phi(s_i) = (1, \ldots , 1)$ for all $i$.
    \end{proof}

\subsection{Reflection generated edge-reversals}

    Given a reflection scheme $\phi$, we now describe what it means for an edge-reversal on a reflective tiling to be generated by $\phi$.

    \begin{defi}\label{def:phigenerated}
        For $\phi$ a $\Gamma$-reflection scheme on $\Dm$, an edge-reversal $\tau \colon T \to \Lambda$ on a reflective $\{m,n\}$ tiling $T$ is $\phi$-generated if for every tile $x \in T_2$, $\tau \circ \bar x = \phi \cdot \overline{\tau(x)}$.
        
        In other words, the right side of the diagram in \eqref{eqn.phigen} commutes.
    \begin{equation}\label{eqn.phigen}
    \begin{tikzcd}[ampersand replacement=\&]
        y(2) \arrow[rr,"s_0"] \arrow[dr,"x"] \arrow[ddr, bend right, "\tau(x)"'] \&\& \P \arrow[ld, "\bar x"'] \arrow[ddl, bend left, "\phi \cdot \overline{\tau(x)}"] \\
        \& T \arrow[d,"\tau"]\& \\
        \& \Lambda \&
    \end{tikzcd}
    \end{equation}
    \end{defi}

    \begin{ex}
        We now provide an example of a $\phi$-generated tiling. First consider $T$, the reflective tiling on $\mathcal{C}_4^\circlearrowright$, as seen in Figure~\ref{fig:reflective/TranslationTiling}. We then define an edge reversal, $\tau:T \to \Lambda$, visualized in Figure~\ref{fig:reflective/TranslationTiling}. Here we see for $e\in T_1$, $\tau(e) = 1$, if the edges are directed in the same manner, and $-1$ otherwise.

        \begin{figure}[h]
            \begin{minipage}{.5\textwidth}
                \centering
                \includegraphics[width=.7\textwidth]{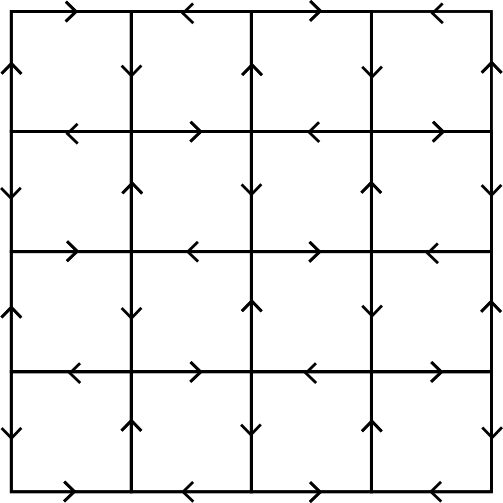}
            \end{minipage}%
            \begin{minipage}{.5\textwidth}
                \centering
                \includegraphics[width=.7\textwidth]{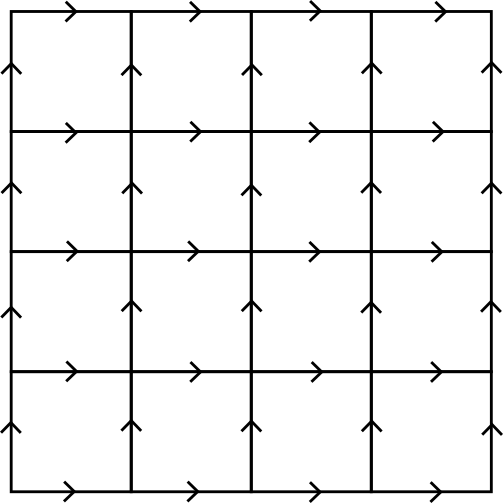}
            \end{minipage}%
            \caption{Visualization of $T:(\mathcal{C}_4^\circlearrowright)^{op} \to \mathrm{Set}$ (left), and an edge-reversal map $\tau:T\to \Lambda$ (right), corresponding to a translation-invariant tiling.}
            \label{fig:reflective/TranslationTiling}
        \end{figure}

        Next, we choose $x:y(2) \to T$ to be the central tile shaded in Figure~\ref{fig:xAndxbar}. Moreover, the other shaded tiles in Figure~\ref{fig:xAndxbar} display the extension, $\overline{x}$.
        \begin{figure}[h]
            \begin{minipage}{.5\textwidth}
                \centering
                \includegraphics[width=.7\textwidth]{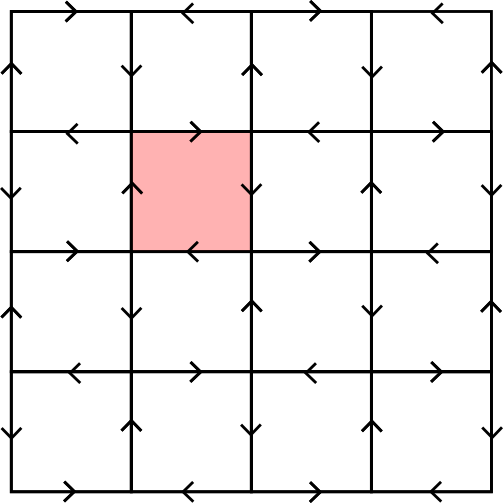}
            \end{minipage}%
            \begin{minipage}{.5\textwidth}
                \centering
                \includegraphics[width=.7\textwidth]{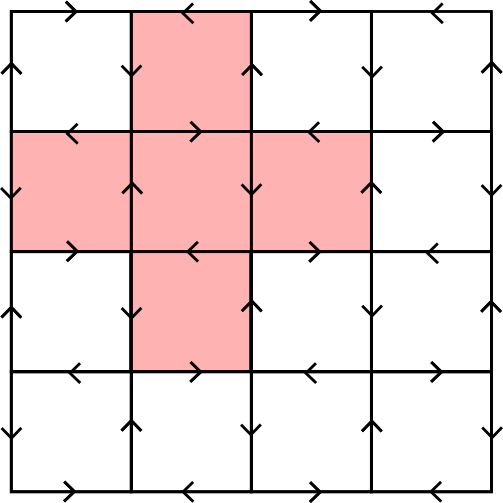}
            \end{minipage}%
            \caption{A choice of tile in $T$ and its extension to $\P$.}
            \label{fig:xAndxbar}
        \end{figure}
        We then see $\tau(x)$ is the tuple, $(1,-1,-1,1)$ and $\overline{\tau(x)}$ maps to the set of tuples, $\{(1,-1,-1,1), (-1,1,1,-1)\}$. In Figure~\ref{fig:tau(x)andtaubar(x)}, one can see a visual representation of both $\tau(x)$ and $\overline{\tau(x)}$.

        \begin{figure}[h]
            \begin{minipage}{.5\textwidth}
            \centering
               \includegraphics[width=.45\textwidth]{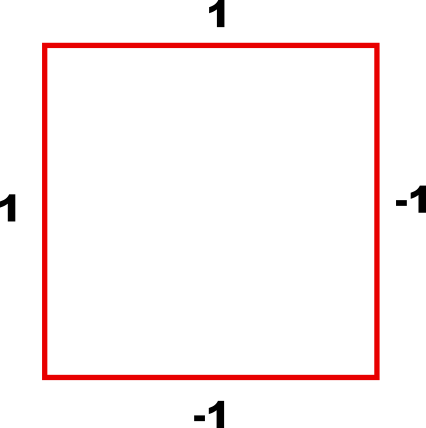} 
            \end{minipage}%
            \begin{minipage}{.5\textwidth}
                \centering
                \includegraphics[width=.8\textwidth]{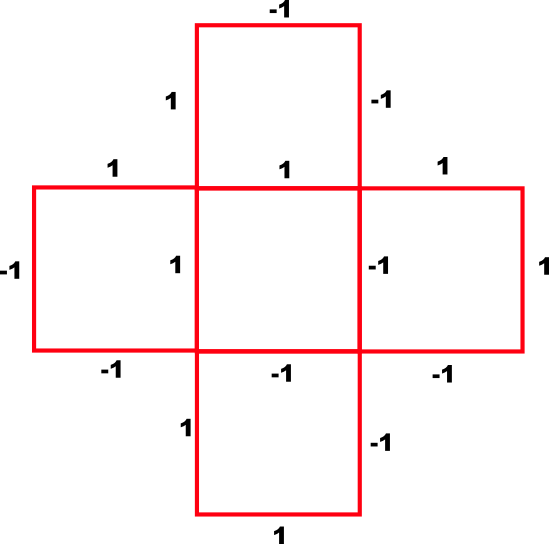}
            \end{minipage}%
            \caption{$\tau(x)$ (left) and $\overline{\tau(x)}$ (right) }
            \label{fig:tau(x)andtaubar(x)}
        \end{figure}

        Now let $\phi_x$ be the map such that such that $\phi_x\cdot \overline{\tau(x)} = \tau(\overline{x})$. We can recover this map based on the group structure of $\Lambda$, and it automatically satisfies \cref{reflectionscheme} as the entire dihedral group is reversal-closed for $m=4$ (see \cref{squaresubset}).

        Observe, by comparing $\tau(\overline{x})$ and $\overline{\tau(x)}$, we get \[\phi_x = \{(1,1,1,1), (1,-1,1,-1), (-1,1,-1,1), (1,-1,1,-1), (-1,1,-1,1)\},\] visually represented in Figure~\ref{fig:tau(xbar)vstau(x)bar}.

        \begin{figure}
        \centering
        \begin{minipage}{.5\textwidth}
          \centering
          \includegraphics[width=.8\linewidth]{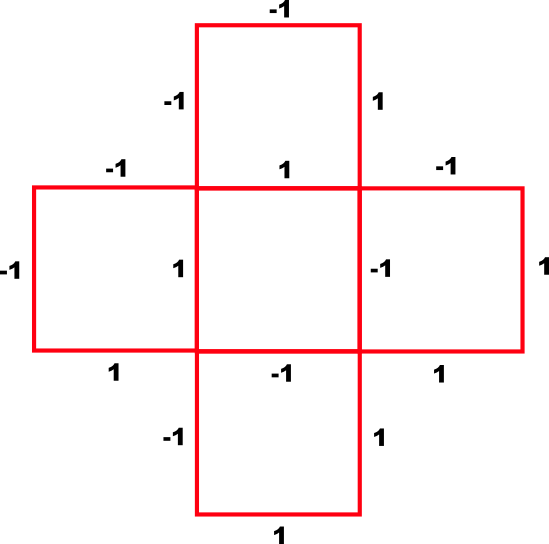}
        \end{minipage}%
        \begin{minipage}{.5\textwidth}
          \centering
          \includegraphics[width=.8\linewidth]{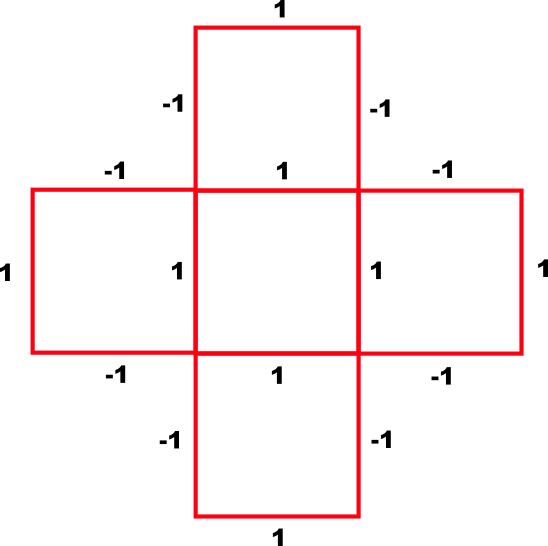}
        \end{minipage}%
        \caption{$\tau(\overline{x})$ (left) and a visualization of $\phi_x$, achieved from comparing $\tau(\overline{x})$ with $\overline{\tau(x)}$.}
        \label{fig:tau(xbar)vstau(x)bar}
        \end{figure}
        
        Because of the symmetries in the image of $\tau$, we see that this map $\phi_x$ is independent of the choice of $x$, thus there exists a unique map, $\phi$ satisfying the commutative diagram in Definition~\ref{def:phigenerated}. In this case, $\Lambda$ is is closed under dihedral actions, thus we can apply Theorem~\ref{newtiling}, and recover an alignment that is also $\phi$-generated.
    \end{ex}

    We now state the main result of this section, which guarantees the existence of $\phi$-generated edge-reversals.

    \begin{theorem}\label{thm:generatetiling}
        For $\phi$ a $\Gamma$-reflection scheme on $\Dm$, $T$ the reflective $\{m,n\}$ tiling on $\Dm$, $x_0 \in T_2$ a tile, and $\sigma \in \Gamma$, there exists a unique $\phi$-generated edge-reversal $\tau \colon T \to \Lambda_{\Dpm}$ with $\tau(x_0) = \deldm \cdot \sigma(\deldpm)$.
    \end{theorem}

    \begin{proof}
        We can define $\tau$ on tiles by $\tau(x) = \deldm \cdot \sigma(\deldpm) \cdot \phi(w)$, where $w$ is the unique element of the Coxeter group $W_{m,n}$ which sends $x_0$ to $x$ and $\phi$ is regarded as a homomorphism $W_{m,n} \to \Lambda_2$. This assignment extends to a map of presheaves $T \to \Lambda$ as the Coxeter group elements corresponding to tiles adjacent along their $d_i$ edge differ by a generator $s_i$ of $W_{m,n}$ and $\phi(s_i)$ has a 1 as its $i$-th component, ensuring that the $i$-th components of $\tau$ applied to both tiles agree. 
        
        It remains then to check that each $\tau(x) \in \deldm \cdot \left[\deldpm\right]_{\Dim}$, and in fact we will show that $\tau(x) = \deldm \cdot \sigma'(\deldpm)$ for some $\sigma' \in \Gamma$. As a base case this is true for $x=x_0$. To use induction, assume that the claim is true when the Coxeter group element $w$ corresponding to a tile is the product of at most $k$ generators $s_i$. If a tile $x$ corresponds to the Coxeter group element $ws_i$ and $$\deldm \cdot \sigma(\deldpm) \cdot \phi(w) = \deldm \cdot \sigma'(\deldpm)$$ for $\sigma' \in \Gamma$ by the inductive hypothesis, then 
        $$\tau(x) = \deldm \cdot \sigma'(\deldpm) \cdot \phi(s_i).$$
        By definition of a $\Gamma$-reflection scheme $\phi(s_i) = \deldpm \cdot \sigma''(\deldpm)$ for some $\sigma'' \in \Gamma$, so 
        $$\tau(x) = \deldm \cdot \sigma'(\deldpm) \cdot \deldpm \cdot \sigma''(\deldpm) = \deldm \cdot \sigma'''(\deldpm)$$
        for some $\sigma''' \in \Gamma$ as $\Gamma$ is reversal-closed. Hence by induction, the proof is complete.
    \end{proof}

    \begin{ex}
    Consider the $\mathcal{C}_4$ category with source and target relations defined as follows:
    \[
        s \circ d^1 = t \circ d^2, t \circ d^1 = s \circ d^4,s \circ d^2 = s \circ d^3, t \circ d^3 = t \circ d^4
    \]
    Then let $\phi$ be defined by the following components:
    \begin{align*}
        \phi(s_0) &= (1,1,1,1)\\
        \phi(s_1) &= (1, 1, -1, -1)\\
        \phi(s_2) &= (1, 1, -1, -1)\\
        \phi(s_3) &= (-1, -1, 1, 1)\\
        \phi(s_4) &= (-1, -1, 1, 1)\\  
    \end{align*}
    
    Now we fix a tile $x_0$ and its image under $\tau$. Let $x_0$ be visualized in Figure~\ref{fig:refTilingGenEx} and $\tau(x_0) = (1,1,1,1)$.

    \begin{figure}[h]
        \begin{minipage}{.5\textwidth}
            \centering
            \includegraphics[width = .4\textwidth]{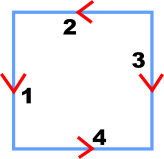}
        \end{minipage}%
        \begin{minipage}{.5\textwidth}
            \centering
            \includegraphics[width=.9\textwidth]{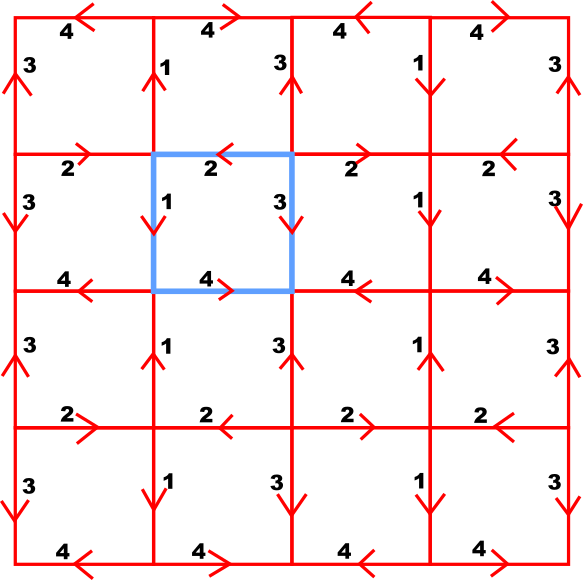}
        \end{minipage}%
        \caption{$T$ a reflective tiling (right) with $x_0$ a tile, highlighted in blue, given from representable (left).}
        \label{fig:refTilingGenEx}
    \end{figure}
    We then define $\tau$ on the adjacent tiles of $x_0$ such that $\tau(\overline{x_0}) = \phi\cdot \overline{\tau(x_0)}$. We then repeat this process to ultimately recover $\tau$. 

    One can see that for two adjacent tiles sharing an edge labeled $i$, $\phi(s_i)$ describes how these two tiles differ. Figure~\ref{fig:GenPlusFromPhi} is an example of the first step in building the tiling out from the image of this fixed tile $x_0$ under $\tau$. For all $i$ we see that the tile reflected across the edge labeled $i$ has red arrows where the tile agrees with the fixed central tile reflected over the edge labeled $i$, and black arrows where it does not. These correspond to the 1's and -1's in $\phi(s_i)$'s.

    \begin{figure}[h]
        \includegraphics[scale=.4]{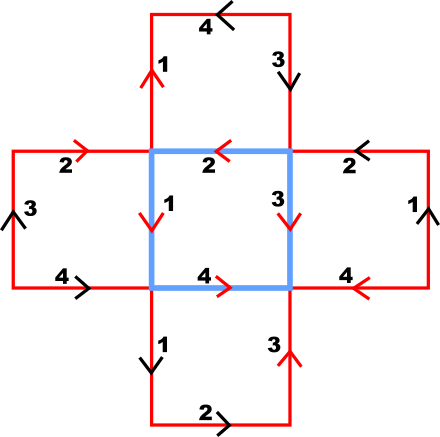}
        \caption{Visualization of adjacent tiles by applying $\phi$ to the image of $x_0$ under $\tau$.}
        \label{fig:GenPlusFromPhi}
    \end{figure}

    We can continue building outwards by attaching a new square at every edge $i$, and directing it according to $\phi(s_i)$, as seen in Figure~\ref{fig:FinalGenTiling}. From Theorem~\ref{thm:generatetiling} this is well-defined. Below is a continuation of this construction process with more tiles. The red arrows are where the directions agree with the reflective tiling given by $\mathcal{C}_4$, and the black arrows are where it does not. Notice the symmetry along diagonals.

    \begin{figure}[h]
        \includegraphics[scale=.4]{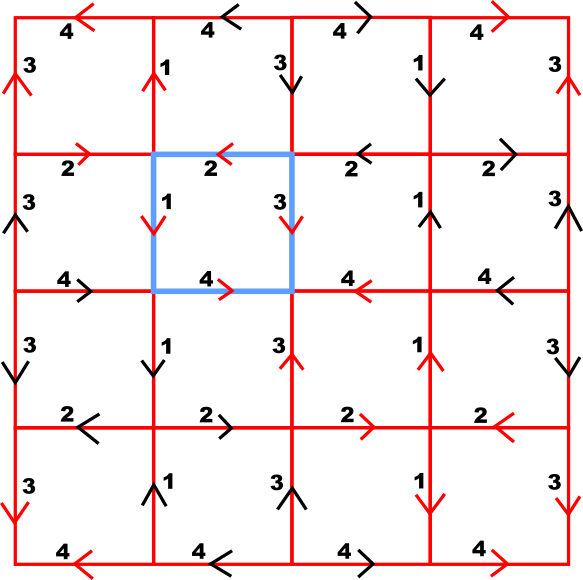}
        \caption{Visualization of edge-reversal map as generated by $\phi$.}
        \label{fig:FinalGenTiling}
    \end{figure}
    
    \end{ex}

    \begin{ex}\label{pentagonreflectiongenerated}
        In Figure~\ref{fig:[pentagonphiexample]} is an example of a reflection generated directed tiling coming from a $\mathcal{C}_5$ category. Here, $\phi_1=(1,1,-1,-1,-1)$, $\phi_2=(1,1,1,1,1)$, $\phi_3=\phi_4=\phi_5=(-1,-1,1,1,1)$. We can check that this choice of $\phi$ is reversal closed. 

        \begin{figure}
            \centering
            \includegraphics[width=0.5\linewidth]{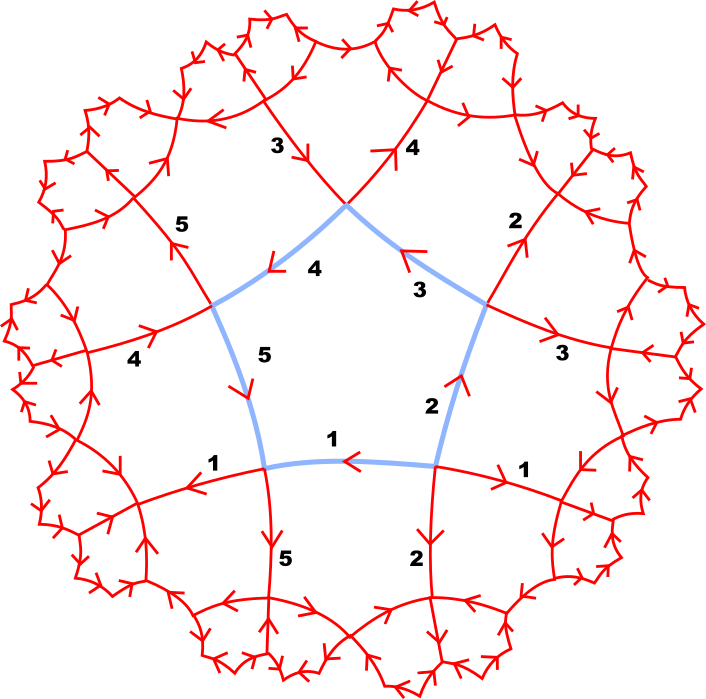}
            \caption{A directed $\{5,4\}$ tiling that is reflection generated.}
            \label{fig:[pentagonphiexample]}
        \end{figure}
    \end{ex}

\section{Global transformation properties of reflection-generated tilings}

    In a geometric $\{m,n\}$ tiling with $n$ even, sequences of opposite edges at each vertex together form geodesic lines through the (hyperbolic) plane, and reflection across such a geodesic preserves the tiling in the sense that vertices are sent to vertices, edges to edges, and tiles to tiles. As such, it is reasonable to expect that under certain conditions such a reflection will also provide an automorphism of a directed $\{m,n\}$ tiling, meaning it additionally preserves the alignment data. We can also show how for a reflection-generated tiling, the failure of such a reflection to preserve the alignment can be precisely measured based on its generating reflection scheme.



    In order to keep track of how the directionality of a tile relates to that of its image under reflection, we will use a track between the two which is carefully chosen to be preserved by the reflection.

    \begin{defi}
         A track is called reflective if its route is a palindrome (as in, invariant under reversal).
    \end{defi}
    
    \begin{lemma}\label{lem:gammaOnTLemma}
        Let $T$ be a reflective tiling. Then the reflection about any geodesic of the undirected graph of $T$ corresponds to an automorphism of $T$.
    \end{lemma}


    \begin{proof}
        Let $T$ be a reflective tiling and $\gamma$ a reflection automorphism of the undirected graph of $T$ across a geodesic. As $\gamma$ is an involution, the functions $\gamma_0,\gamma_1,\gamma_2$ on vertices, edges, and tiles are all bijections, but we need to show that $\gamma$ is a map of presheaves, making the diagram below commute for each structure map $s,t,d_1,...,d_m$.
        \[
        \begin{tikzcd}
            T_0 \arrow[dd, "\gamma_0"]  &
            T_1\arrow[dd, "\gamma_1"] \arrow[l, shift left = .75ex, "t"] \arrow[l, shift right = .75ex, "s"'] &
            T_2 \arrow[dd, "\gamma_2"]\arrow[l, draw=none, shift right = 2.5ex, "\vdots", "d_1"']  \arrow[l, shift right = 2.5ex] \arrow[l, shift left = 2.5ex, "d_m"]\\\\
            T_0'  &
            T_1' \arrow[l, shift left = .75ex, "t"] \arrow[l, shift right = .75ex, "s"'] &
            T_2' \arrow[l, draw=none, shift right = 2.5ex, "\vdots", "d_1"']  \arrow[l, shift right = 2.5ex] \arrow[l, shift left = 2.5ex, "d_m"]
        \end{tikzcd}
        \]

        We show that $\gamma_1(d_i(x)) = d_i(\gamma_2(x))$ by induction on the distance between $x$ and $\gamma_2(x)$. First suppose $x$ and $\gamma_2(x)$ are adjacent, which as $T$ is reflective means that $d_i(x) = d_i(\gamma_2(x))$ for exactly one $i$. This edge must then lie on the geodesic which is fixed by $\gamma_1$, so we have $\gamma_1(d_i(x)) = d_i(x) = d_i(\gamma_2(x))$, which by the construction of the reflective tiling in \cref{reflectiveexists} implies that $\gamma_1(d_j(x)) = d_j(\gamma_2(x))$ as the edge labels in adjacent tilings are reflections of one another across their shared edge. This also shows why moreover $s(\gamma_1(e))=\gamma_0(s(e))$ and $t(\gamma_1(e))=\gamma_0(t(e))$ for all of the edges $e$ of the tile $x$: the direction of $d_i(x)$ is fixed as its vertices are, and the remaining edges' directions are determined by this one.

        Note that in this case there is a length 1 reflective track from $x$ to $\gamma_2(x)$. We will argue for the remaining tiles by induction on the length of a track from $x$ to $\gamma_2(x)$. It can be assumed that such a track is reflective: any given track must pass through the geodesic axis of reflection, and as soon as a track does so its remaining tiles can be replaced with reflections of those from $x$ to the geodesic.

        Assume that for tiles $x$ such that $x$ and $\gamma_2(x)$ have a track between them of length at most $2k-1$, and for all of their edges, $\gamma$ preserves the edge and endpoint labels. Suppose $x$ and $\gamma_2(x)$ admit a reflective track of length $2k+1$ which we denote by $x=a_0,a_1,\ldots,a_k, b_k, b_{k-1}, \ldots b_0=\gamma_2(x)$, where $\gamma_2(a_i) = b_i$. By assumption, we have $\gamma_1(d_i(a_1)) = d_i(\gamma_2(a_1)) = d_i(b_1)$ for all $i$. We can then deduce that for $j$ the label of the edge between $a_0$ and $a_1$,
        $$\gamma_1(d_j(a_0)) = \gamma_1(d_j(a_1)) = d_j(b_1) = d_j(b_0)$$
        as desired and $\gamma_1$ also preserves the endpoints of the edges $d_j(a_0)$ and $d_j(b_0)$. This determines that the edge labels and directions of the remaining edges of $x=a_0$ will also be preserved by $\gamma$, completing the proof.
    \end{proof}

    While this shows that in a reflective tiling all of the geometric reflections act as automorphisms, in more general tilings this will not typically be the case.    

    \begin{ex}
        First consider the square category $\mathcal{C}_4$ modeling the directed square tile in Figure~\ref{fig:C4foruniqueglobalreflection},
        \begin{figure}[h]
            \centering
            \includegraphics[width=.1\linewidth]{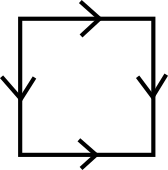}
            \caption{Visualization of relations in $\mathcal{C}_4$ category.}
            \label{fig:C4foruniqueglobalreflection}
        \end{figure}
and the directed tiling on this category in Figure~\ref{fig:uniqueglobalreflection}.
        \begin{figure}[h]
            \centering
            \includegraphics[width=0.5\linewidth]{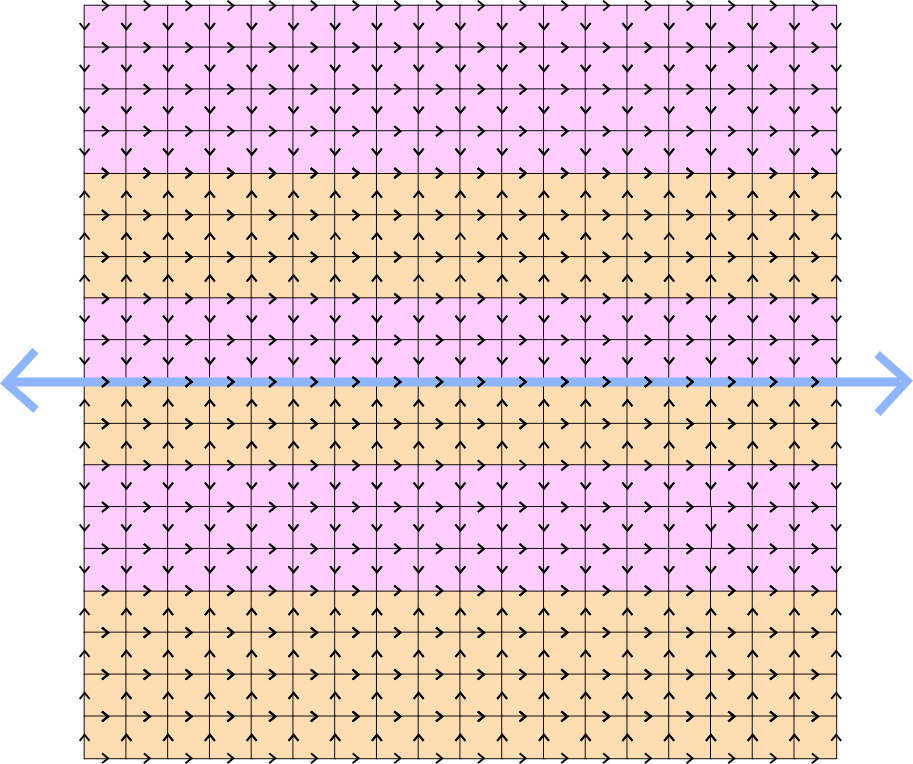}
            \caption{Directed tiling with unique horizontal reflection axis highlighted in blue.}
            \label{fig:uniqueglobalreflection}
        \end{figure}

        Observe that for two adjacent tiles of different colors that share an edge $e$, the edge directions of those two tiles are preserved by reflection across $e$. Similarly, two tiles of the same color sharing an edge indicates that translation across that edge preserves the tile alignments. The tiling is constructed such that the number of adjacent rows with squares of the same color increases by 1 with each alternating batch. This ensures that there is only 1 horizontal line of edges, highlighted in blue, across which reflection preserves the alignment. Reflection across any other horizontal axis would not be an automorphism of the directed tiling. 
    \end{ex}

    We might however expect more convenient properties to hold for reflection-generated tilings, given their close relationship to the reflective tiling.

    \begin{ex}
        In Figure~\ref{fig:globalreflectionex} is the reflection generated tiling $T$ from \cref{pentagonreflectiongenerated}, where $\phi_1=(1,1,-1,-1,-1)$, $\phi_2=(1,1,1,1,1)$, and $\phi_3=\phi_4=\phi_5=(-1,-1,1,1,1)$. The two pink tiles are reflections of one another over the geodesic labelled $g_1$ whose edges are labeled $d_1$ in the reflective base tiling used to generate $T$. Observe that the edge reversal code between the two tiles after reflection is precisely $\phi_1$.
        

        \begin{figure}[h]
        \includegraphics[scale=.3]{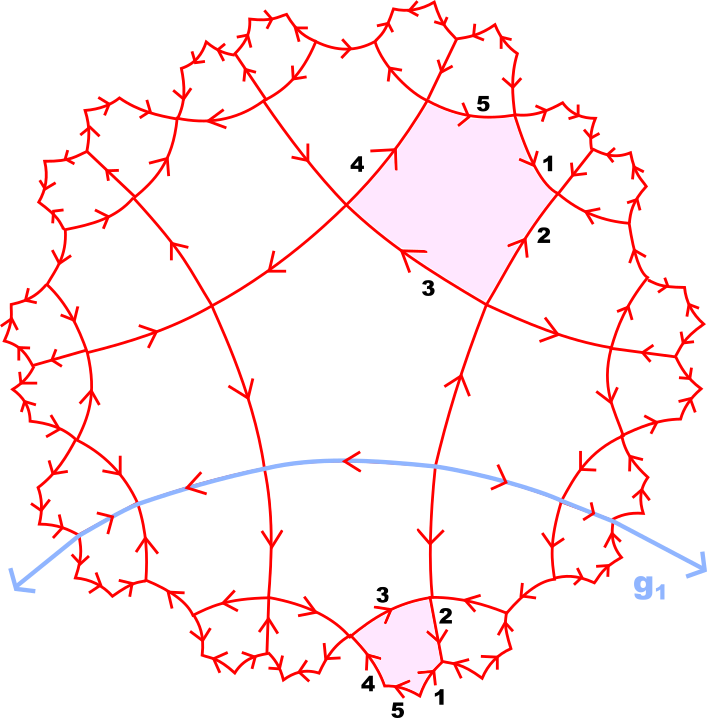}
        \caption{Two tiles related by reflection across the geodesic $g_1$ in a reflection generated $\{5,4\}$ tiling.}
        \label{fig:globalreflectionex}
    \end{figure}
    \end{ex}
    
    The reflection schema describe the local properties of a tiling in terms of how the directions of one tile relate to those of its adjacent tiles. But as in the previous example, these local relationships extend to describe how these tilings interact with the global action of reflecting along a geodesic. In fact, we can describe more generally how any edge-reversal $\tau$ interacts with such a reflection automorphism on the underlying reflective tiling.
    
    \begin{defi}
        Let $T$ be a directed tiling on $\Dm$, $\tau: T \to \Lambda_{\Dpm}$, $\gamma$ a reflection automorphism of $T$, and $\psi_\gamma \in (\Lambda_{\Dm})_2$. The edge-reversal $\tau$ is $\psi_\gamma$-reflective with respect to $\gamma$ if the diagram below commutes,
    
    \[
    \begin{tikzcd}
    T_2 \arrow[d, "\tau"] \arrow[rr, "\gamma"]& & T_2\arrow[d,"\tau"]\\
    (\Lambda_{\Dpm})_2 \arrow[rr, "\cdot\psi_\gamma"] & & (\Lambda_{\Dpm})_2
    \end{tikzcd}
    \]
    where the arrow $\cdot\psi_\gamma$ denotes multiplication by $\psi_\gamma$.
    
    This ensures that for any tile $x \in T_2$, $\tau(\gamma(x)) = \psi_\gamma(\tau(x))$, so that $\tau$ is invariant under the reflection $\gamma$ up to the fixed action of $\psi_\gamma$.
    \end{defi}

    \begin{theorem}\label{globalreflections}
         An edge-reversal $\tau \colon T \to \Lambda_{\Dpm}$ from a reflective tiling is reflection-generated (for $n = 0$ (mod 4)) or reflective (for $n = 2$ (mod 4) if and only if each reflection $\gamma$ across a geodesic in $T$ has an associated $\psi_\gamma$ such that $\tau$ is $\psi_\gamma$-reflective with respect to $\gamma$. In particular, if $\gamma$ reflects across a geodesic made up of edges labeled $d_i$ and $\tau$ is generated by the reflection scheme $\phi$, then $\psi_\gamma = \phi_i$.
    \end{theorem}

    \begin{proof}
        First suppose $T$ is a reflective tiling and $\tau: T \to \Lambda_{\Dpm}$ an edge reversal map that is reflection generated by $\phi$. Now let $\gamma$ be a reflection across a fixed geodesic in $T$. Suppose $x\in T_2$ such that $\gamma(x)$ is adjacent to $x$. Then for some $1\leq i \leq m$, $d_i(x) = d_i(\gamma(x))$. Define $\psi_\gamma = \phi(s_i)$. By definition of reflection generation, $\psi_\gamma \cdot \tau(x) = \tau(\gamma(x))$. 
        
        Now suppose $y$ is an arbitrary tile in $T_2$ and let $x$ be a tile adjacent to the geodesic. As $T$ is reflective, there is some track $y=a_0,\ldots, a_k = x$ from $y$ to $x$ and therefore a reflective track $$y=a_0,\ldots,a_k,\gamma(a_k),\ldots,\gamma(a_0)=\gamma(y)$$
        of length $2k+1$, as $\gamma(a_k)=\gamma(x)$ is adjacent to $x=a_k$ and $\gamma$ preserves routes. In particular, the route of this track is the route of the track $a_0,\ldots,a_k$ followed by $i$ followed by the reverse of the route of $a_0,\ldots,a_k$.
        
        We will proceed by induction on $k$, with the base case $k=0$ shown above. Suppose the diagram commutes for tracks of length less than or equal to $2k+1$ and consider the track, $$y=a_0,\ldots,a_{k+1},\gamma(a_{k+1}),\ldots,\gamma(a_0)=\gamma(y)$$ 
        of length $2k+3$. Observe, the subtrack $a_1,\ldots,\gamma(a_1)$ is a track of length $2k+1$, so $\tau(\gamma(a_1)) = \psi_\gamma(\tau(a_1))$. But observe, $y$ is adjacent to $a_1$ along the edge $d_j(y)$ and by \cref{lem:gammaOnTLemma} $\gamma(y)$ is adjacent to $\gamma(a_1)$ on their $j$-th edges, so  we have 
        $$\tau(\gamma(y)) = \phi_j \cdot \tau(\gamma(a_1)) = \phi_j \cdot \psi_\gamma \cdot \tau(a_1) = \psi_\gamma \cdot \phi_j \cdot \tau(a_1) = \psi_\gamma \cdot \tau(y).$$ 
        Therefore, as $\psi_\gamma \cdot \tau(y) = \tau(\gamma(y))$ for all $y \in T_2$, $\tau$ is $\psi_\gamma$-reflective.

        Now for the converse suppose $\tau$ is $\psi_\gamma$ reflective for all $\gamma$. First suppose $n$ is divisible by 4. Then each geodesic is made up of edges which are $d_i$ of their adjacent tiles for some fixed $i \in \{1,...,m\}$. 

        First observe if $\gamma_1$ and $\gamma_2$ are both labeled $i$, then $\psi_{\gamma_1} = \psi_{\gamma_2}$. Because these geodesics have the same label, then there must be some reflection $\eta$ such that $\gamma_2\eta\gamma_1 = \eta$. 
        Then from the definition of $\tau$ being reflective with respect to every geodesic, the following diagram commutes.
        \[
        \begin{tikzcd}
            T_2 \arrow[d, "\tau"] \arrow[rr, "\gamma_1"]& & T_2 \arrow[d, "\tau"] \arrow[rr, "\eta"]& & T_2 \arrow[d, "\tau"] \arrow[rr, "\gamma_2"]& & T_2\arrow[d,"\tau"]\\
            (\Lambda_{\Dpm})_2 \arrow[rr, "\cdot\psi_{\gamma_1}"] & &(\Lambda_{\Dpm})_2 \arrow[rr, "\cdot\psi_\eta"] & & (\Lambda_{\Dpm})_2 \arrow[rr, "\cdot\psi_{\gamma_2}"] & & (\Lambda_{\Dpm})_2
        \end{tikzcd}
        \]
        But as $\gamma_2 \eta \gamma_1 = \eta$, we get from this that $\psi_{\gamma_2}\psi_{\eta}\psi_{\gamma_1} = \psi_\eta$. As the group $\Lambda_2$ is abelian and every element has order 2, we get that $\psi_{\eta}\psi_{\gamma_1}=\psi_{\gamma_2}\psi_{\eta}$ and therefore $\psi_{\gamma_1} = \psi_{\gamma_2}$.

        Now suppose $n \equiv 2 (\textrm{mod }4)$. It suffices to show that for each reflection $\gamma$ of $T$, the corresponding element $\psi_\gamma$ is trivial. Let $x \in  T_2$ be a tile such that $\gamma(x)$ is adjacent to $x$. Then consider the sequence of geodesics, $\gamma = \gamma_1, \gamma_2,\ldots, \gamma_n$, that correspond to reflecting $x$ about all of the edges adjacent to one of its vertices, beginning with $\gamma$. But observe that $\gamma_{i}\gamma_{i+1}\gamma_{i+2} = \gamma_{i+1}$: as $\gamma_{i+1}$ sends the $i$-th edge to the $(i+1)$-th edge we have $\gamma_{i}\gamma_{i+1}=\gamma_{i+2}\gamma_{i+1}$. This shows that $\psi_{\gamma_{i}}\psi_{\gamma_{i+1}}\psi_{\gamma_{i+2}} = \psi_{\gamma_{i+1}}$, which as before implies that $\psi_{\gamma_{i}} = \psi_{\gamma_{i+2}}$ for all $i = 1,\ldots, n$. 
        
        Therefore, as $\psi_{\gamma_1}\cdots\psi_{\gamma_n} = (1,\ldots,1)$ and $\psi_{\gamma_1}\cdots\psi_{\gamma_n} = (\psi_{\gamma_1} = \psi_{\gamma_2})^{n/2}$ with $n/2$ odd, we have $\psi_{\gamma_1}\psi_{\gamma_2} = (1,\ldots,1)$. Hence $\psi_{\gamma_1} = \psi_{\gamma_2}$, and so $\psi_{\gamma}$ and $\psi_{\gamma'}$ agree for all reflections $\gamma$ that intersect at some vertex. The same can then be shown for any two geodesics that do not intersect by choosing a third geodesic intersecting both. 
        
        We can then repeat this process and see that for any two geodesics, $\gamma$ and $\eta$, $\psi_\gamma = \psi_\eta$. Now take $x\in T_2$ such that $\gamma(x)$ is adjacent to $x$, then $\psi_\gamma$ must fix at least one edge. In particular, $\psi_\gamma$ has a 1 in some component. But, we can find a geodesic that has to fix any edge, and as $\psi_\gamma$ is the same for every choice of $\gamma$, we see that $\psi_\gamma = \underbrace{(1,1,\ldots,1)}_{\text{m times}}$. Therefore, $T$ is reflective.
    \end{proof}

    The automorphisms of undirected tilings (in the geometric world, isometries preserving the vertices, edges, and tiles) are generated by reflections, and the Coxeter group forms a subgroup of those automorphisms generated by reflections only across geodesics in the underlying graph of the tiling.

    Recall from \cref{coxeterhomomorphism} that a reflection scheme $\phi$ can be regarded as a group homomorphism $W_{m,n} \to \Lambda_2$. This homomorphism tells us how to extend \cref{globalreflections} to arbitrary automorphisms of the $\{m,n\}$ tiling.

    \begin{defi}\label{generalsymmetry}
        Let $T$ be a directed tiling on $\Dm$, $\tau: T \to \Lambda_{\Dpm}$, $\gamma$ any automorphism of $T$, and $\psi_\gamma \in (\Lambda_{\Dm})_2$. The edge-reversal $\tau$ is $\psi_\gamma$-symmetric with respect to $\gamma$ if the diagram below commutes,
    \[
    \begin{tikzcd}
    T_2 \arrow[d, "\tau"] \arrow[rr, "\gamma"]& & T_2\arrow[d,"\tau"]\\
    (\Lambda_{\Dpm})_2 \arrow[rr, "\cdot\psi_\gamma"] & & (\Lambda_{\Dpm})_2
    \end{tikzcd}
    \]
    where the arrow $\cdot\psi_\gamma$ denotes multiplication by $\psi_\gamma$.
    
    This means that $\tau$ is invariant under the automorphism $\gamma$ up to the fixed action of $\psi_\gamma$.
    \end{defi}

    \begin{cor}\label{automorphismphi}
        An edge-reversal $\tau \colon T \to \Lambda_{\Dpm}$ from the reflective $\{m,n\}$ tiling generated by a reflection scheme $\phi$ is $\phi(\gamma)$-symmetric with respect to an automorphism $\gamma$ in the Coxeter group $W_{m,n}$.
    \end{cor}

    \begin{proof}
        This follows directly from \cref{globalreflections}, as $\gamma$ is a composite of reflections and the definition of $\psi_\gamma$-symmetry respects composition: if $\gamma$ is $\psi_\gamma$-symmetric and $\gamma'$ is $\psi_{\gamma'}$-symmetric, then the composite $\gamma\gamma'$ is $\psi_\gamma\psi_{\gamma'}$-symmetric. As the generating reflection $s_i$ is $\phi(s_i)$-symmetric for all $i$, this shows that any $\gamma$ in the Coxeter group is $\phi(\gamma)$-symmetric.
    \end{proof}

    This corollary tells us how to compare the directions of tiles under all of the standard orientation preserving automorphisms of the (hyperbolic) plane: translations, rotations, and (in the hyperbolic case) parabolic transformations.

    \begin{ex}
       Figure \ref{fig:globaltranslationexample} shows a global translation. On the left, we have a reflection-generated $\{5,4\}$ tiling given by $\phi_1=(1,1,-1,-1,-1)$, $\phi_2=(1,1,1,1,1)$, $\phi_3=\phi_4=\phi_5=(-1,-1,1,1,1)$. Two non-intersecting geodesics are highlighted in blue, $g_5$ and $g_2$ being labeled 5 and 2 respectively. We translate the entire tiling by reflecting over $g_5$ and then $g_2$ to obtain the directed tiling on the right (shown by the location of the colored tiles). Notice that we can express how the arrows change during this translation by changing every edge labeled $1$ or $2$. This is because $g_2 \circ g_5$ corresponds to multiplying by $(1,1,1,1,1)(-1,-1,1,1,1)=(-1,-1,1,1,1)$. 

        \begin{figure}
        \centering
        \begin{minipage}{.5\textwidth}
          \centering
          \includegraphics[width=.8\linewidth]{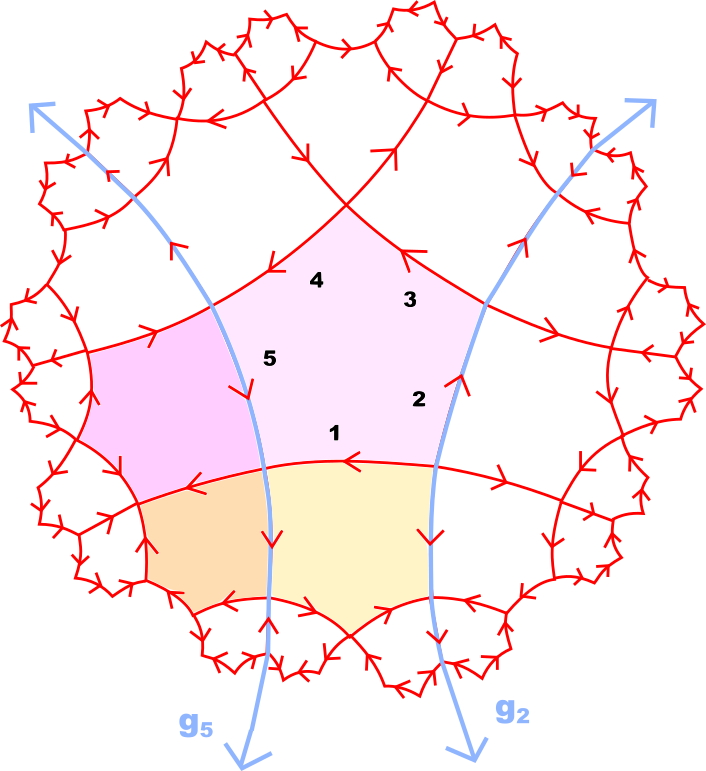}
        \end{minipage}%
        \begin{minipage}{.5\textwidth}
          \centering
          \includegraphics[width=.8\linewidth]{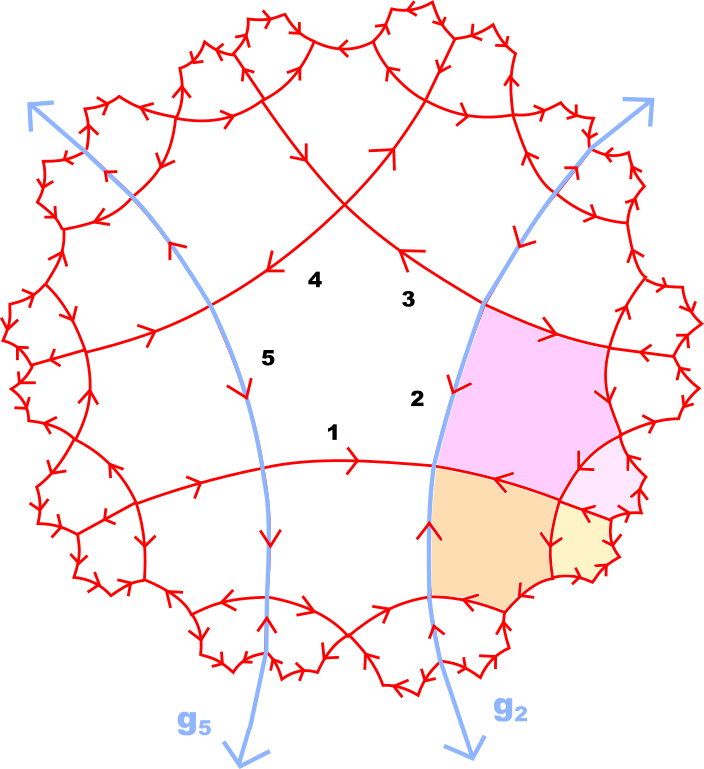}
        \end{minipage}%
        \caption{A global translation.}
        \label{fig:globaltranslationexample}
        \end{figure}
    \end{ex}

    \begin{ex}[Global Rotation]
        \begin{figure}
        \centering
        \begin{minipage}{.5\textwidth}
          \centering
          \includegraphics[width=.8\linewidth]{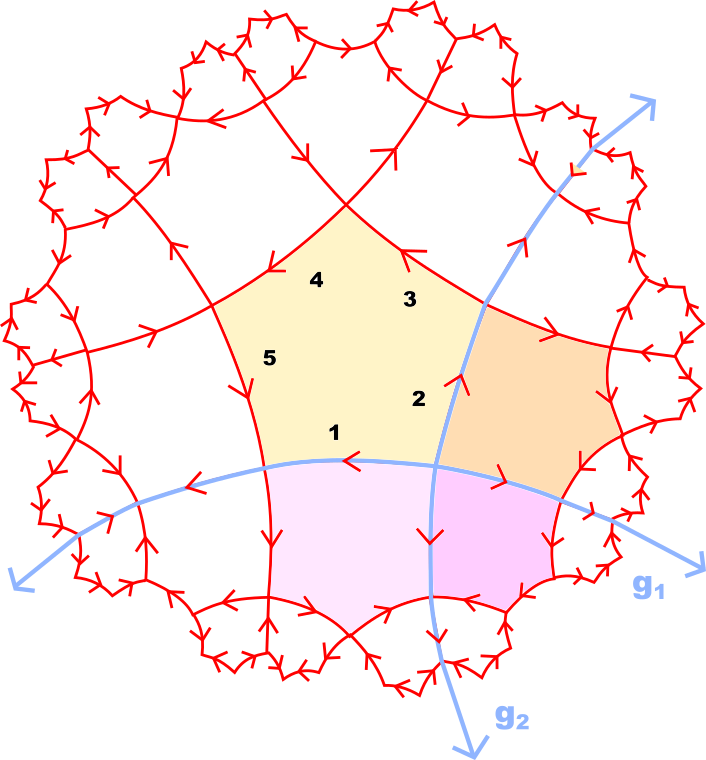}
        \end{minipage}%
        \begin{minipage}{.5\textwidth}
          \centering
          \includegraphics[width=.8\linewidth]{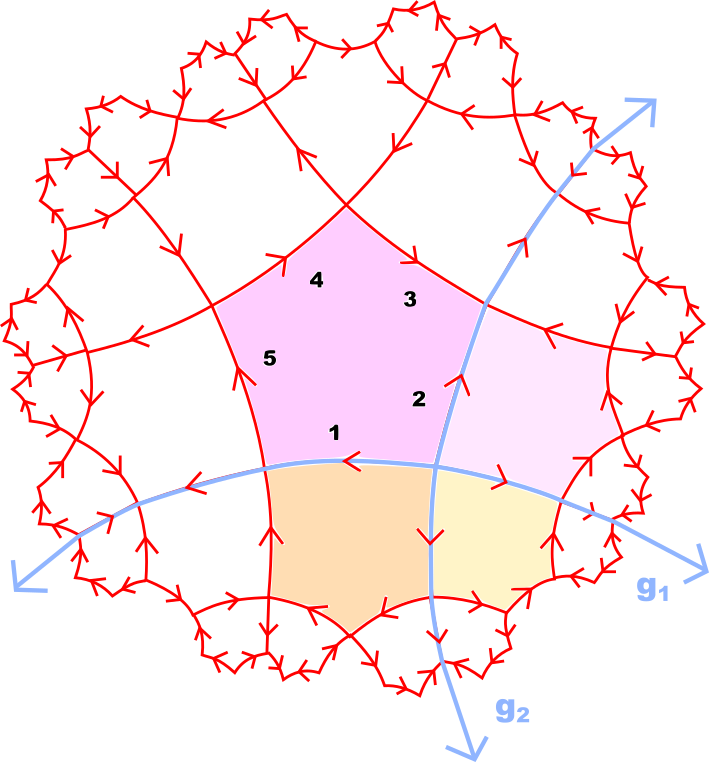}
        \end{minipage}%
        \caption{A global rotation.}
        \label{fig:globalrotationexample}
        \end{figure}

         Figure \ref{fig:globalrotationexample} shows a global rotation. On the left, we have the same $\{5,4\}$ tiling shown above. Two intersecting geodesics are highlighted in blue, $g_1$ and $g_2$. We rotate the entire tiling by reflecting over $g_1$ and then $g_2$ to obtain the directed tiling on the right. Similarly, we can express how the arrows change by during this translation by changing every edge labeled 3, 4 or 5, as $g_2 \circ g_1$ corresponds to multiplying by $(1,1,1,1,1)(1,1,-1,-1,-1)=(1,1,-1,-1,-1)$. Notice that with rotation, any rotation of the same ``angle'' about a vertex with the same label will have the same $\phi_\gamma$, since we can conjugate by translation, then cancel the conjugation as the group of tuples is abelian.
    \end{ex}

\bibliography{references}
\bibliographystyle{alpha}

\appendix

\section{Reversal-Closed Subsets}\label{appendixreversal}

We now turn to a cursory analysis of how to find reversal-closed subsets for a fixed $\delta \in \{1,-1\}^m$. The maximal reversal-closed subsets can be computed by enumerating the conflicts between elements of $\Dim$ ($\sigma_1,\sigma_2$ cannot coexist in a reversal-closed subset if $\delta \cdot \sigma_1(\delta) \cdot \sigma_2(\delta) \neq \sigma_3(\delta)$ for all $\sigma_3 \in \Dim$), enumerating the maximal conflict-free subsets of $\Dim$, and then filtering these subsets by whether they satisfy the reversal-closed condition (if $\sigma_1,\sigma_2 \in \Gamma$ for such a subset $\Gamma$, then at least one $\sigma_3 \in \Dim$ such that $\delta \cdot \sigma_1(\delta) \cdot \sigma_2(\delta) = \sigma_3(\delta)$ must be included in $\Gamma$ otherwise $\Gamma$ is discarded). 

While a more detailed algorithm is beyond the scope of this paper, we have computed these subsets for values of $m$ up to 10 and confirmed every tile in this range has such a subset containing at least 4 group elements, with each tile admitting multiple except in cases which we enumerate below where the entire set $\Dim$ is reversal-closed.

\begin{lemma}
    The sets $\{e\}$ and $\{e,\sigma\}$ are reversal closed for any $\sigma \in \Dim$ and any $\delta \in \{1,-1\}^m$, and the identity $e$ belongs to any maximal reversal-closed subset for $\delta$. 
\end{lemma}

\begin{proof}
    The first claim follows from the equations $\delta \cdot \delta \cdot \delta = \delta$ and $\delta \cdot \sigma(\delta) \cdot \sigma(\delta) = \delta$.

    For the second claim, there are no elements $\sigma \in \Dim$ which are prevented from being in the same reversal-closed subset as $e$, as it is always the case that 
    $$\delta \cdot \delta \cdot \sigma(\delta) = \sigma(\delta)$$
    as $\delta \cdot \delta = (1,...,1)$.
\end{proof}

\begin{lemma}
    The entire set $\Dim$ is reversal-closed for $(1,...,1) \in \{1,-1\}^m$.
\end{lemma}

\begin{proof}
    As $\sigma(1,...,1) = \pm (1,...,1)$ for all $\sigma \in \Dim$, we also have $(1,...,1) \cdot \sigma_1(1,...,1) \cdot \sigma_2(1,...,1) = \pm(1,...,1)$ for all $\sigma_1,\sigma_2 \in \Dim$. As $(1,...,1) = e(1,...,1)$ and $-(1,...,1) = f(1,...,1)$, this shows that $\Dim$ is reversal-closed.
\end{proof}

We now compare the reversal-closed subsets for $\delta$ to those for $\sigma(\delta)$.

\begin{lemma}\label{sigmamult}
    For any $\delta,\delta',\delta'' \in \{1,-1\}^m$ and $\sigma \in \Dim$, $$\sigma(\delta) \cdot \sigma(\delta') \cdot \sigma(\delta'') = \sigma(\delta \cdot \delta' \cdot \delta'').$$
\end{lemma}

Note that this is not the case for binary products in $\{1,-1\}^m$ which are not generally respected by the action of $\Dim$. While the generator $r$ does respect the identity and multiplication, it is straightforward to check that $f(1,...,1) = -(1,...,1)$ and $f(\delta \cdot \delta') = -f(\delta) \cdot f(\delta')$. Thus we can conclude that for any $\sigma \in \Dim$, $\sigma(1,...,1) = \pm (1,...,1)$ and $\sigma(\delta \cdot \delta') = \pm \sigma(\delta) \cdot \sigma(\delta')$ with the sign determined by the sign of $\sigma$ under the homomorphism $\Dim \to \{1,-1\}$ sending $r$ to 1 and $f$ to -1.

\begin{proof}
    Based on the discussion above, we can conclude that 
    $$\sigma(\delta \cdot \delta' \cdot \delta'') = \pm \sigma(\delta \cdot \delta') \cdot \sigma(\delta'') = \sigma(\delta) \cdot \sigma(\delta') \cdot \sigma(\delta'')$$
    as the sign introduced by each product depends only on $\sigma$.
\end{proof}

\begin{lemma}\label{sigmarelation}
    For $\delta \in \{1,-1\}^m$ and $\sigma_1,\sigma_2,\sigma_3,\sigma \in \Dim$,
    $$\delta \cdot \sigma_1(\delta) \cdot \sigma_2(\delta) = \sigma_3(\delta)$$
    if any only if 
    $$\sigma(\delta) \cdot \left(\sigma\sigma_1\sigma^{-1}\right)\sigma(\delta) \cdot \left(\sigma\sigma_1\sigma^{-1}\right)\sigma(\delta) = \left(\sigma\sigma_1\sigma^{-1}\right)\sigma(\delta).$$
\end{lemma}

\begin{proof}
    Applying $\sigma$ to the first equation gives us 
    $$\sigma\left(\delta \cdot \sigma_1(\delta) \cdot \sigma_2(\delta)\right) = \sigma\sigma_3(\delta),$$
    and using \cref{sigmamult} shows that
    $$\sigma(\delta) \cdot \sigma\sigma_1(\delta) \cdot \sigma\sigma_2(\delta) = \sigma\sigma_3(\delta).$$
    Applying the identity in the form $\sigma^{-1}\sigma$ to all but the leftmost copy of $\delta$ turns this equation into
    $$\sigma(\delta) \cdot \sigma\sigma_1\sigma^{-1}\sigma(\delta) \cdot \sigma\sigma_2\sigma^{-1}\sigma(\delta) = \sigma\sigma_3\sigma^{-1}\sigma(\delta)$$
    as desired, and each of these steps is reversible.
\end{proof}

This lemma immediately proves the following, as all of the information determining reversal-closure for $\delta$ translates to that for $\sigma(\delta)$ under conjugation by $\sigma$ in $\Dim$.

\begin{prop}\label{reversalconjugate}
    $\Gamma$ is a reversal-closed subset for $\delta$ if and only if $\sigma\Gamma\sigma^{-1} = \{\sigma\tau\sigma^{-1} | \tau \in \Gamma\}$ is a reversal-closed subset for $\sigma(\delta)$.
\end{prop}

In particular, this shows that $\delta$ and $\sigma(\delta)$ have the same number and sizes of maximal reversal-closed subsets. In the examples that follow, we can therefore restrict our attention to a single representative of each orbit in $\{1,-1\}^m$ under $\Dim$.

\begin{ex}\label{trianglesubset}
    When $m=3$, those orbits are represented by $(1,1,1)$ (corresponding to the cyclic triangle) and $(1,-1,1)$ (corresponding to the standard directed triangle). For the cyclic triangle, $\mathsf{D}_3$ itself is reversal-closed, but for the standard triangle the maximal reversal-closed subsets are $\{e,f,fr,rr\}$ and $\{e,f,r,frr\}$: this follows from the equations
    $$f(1,-1,1) \cdot fr(1,-1,1) \cdot rr(1,-1,1) = (-1,1,-1) \cdot (-1,-1,1) \cdot (1,1,-1) = (1,-1,1)$$
    and
    $$f(1,-1,1) \cdot r(1,-1,1) \cdot frr(1,-1,1) = (-1,1,-1) \cdot (-1,1,1) \cdot (1,-1,-1) = (1,-1,1),$$
    and the inability of either $fr$ or $rr$ to be in a reversal-closed subset for $(1,-1,1)$ with either $r$ or $frr$.
\end{ex}

\begin{ex}\label{squaresubset}
    When $m=4$, the orbits are represented by $(1,1,1,1)$, $(1,1,1,-1)$, $(1,1,-1,-1)$, and $(1,-1,1,-1)$, and in each case $\mathsf{D}_4$ is reversal-closed. For instance, when $\delta = (1,1,1,-1)$ this is witnessed by the 7 equations below.
    $$f(\delta) \cdot r(\delta) \cdot rr(\delta) = \delta \qquad
    f(\delta) \cdot fr(\delta) \cdot frr(\delta) = \delta \qquad
    f(\delta) \cdot rrr(\delta) \cdot frrr(\delta) = \delta$$
    $$r(\delta) \cdot fr(\delta) \cdot rrr(\delta) = \delta \qquad 
    r(\delta) \cdot frr(\delta) \cdot frrr(\delta) = \delta$$ 
    $$fr(\delta) \cdot rr(\delta) \cdot frrr(\delta) = \delta \qquad
    rr(\delta) \cdot frr(\delta) \cdot rrr(\delta) = \delta$$
\end{ex}

We can also determine additional reversal-closed subsets based on the symmetry of $\delta \in \{1,-1\}^m$.

\begin{cor}\label{reversalstabilizer}
    If $\Gamma$ is reversal-closed for $\delta \in \{1,-1\}^m$, and $\sigma(\delta) = \delta$, then $\Gamma\sigma$ and $\sigma\Gamma$ are reversal-closed for $\delta$.
\end{cor}

\begin{proof}
    That $\Gamma\sigma$ is reversal closed for $\delta$ follows from the fact that for $\sigma_1,\sigma_2,\sigma_3 \in \Gamma$, 
    $$\sigma_1\sigma(\delta) \cdot \sigma_2\sigma(\delta) \cdot \sigma_2\sigma(\delta) = \sigma_1(\delta) \cdot \sigma_2(\delta) \cdot \sigma_3(\delta) = \delta.$$
    Then by \cref{reversalconjugate}, $\sigma\Gamma = \sigma\Gamma\sigma^{-1}\sigma$ is also reversal-closed.
\end{proof}

\begin{ex}
    For $\delta = (-1,-1,1,-1,1,1)$, $f(\delta) = \delta$. We can therefore conclude that as $\Gamma = \{e,f,frr,rrrr\}$ is a reversal-closed subset of $\mathsf{D}_6$ for $\delta$, so is $f\Gamma = \{e,f,rr,frrrr\}$. In this case $\Gamma f = \Gamma$ as $frrf=rrrr$.
\end{ex}

For $\delta \in \{1,-1\}^m$, there are two ways of constructing an element of $\{1,-1\}^{km}$ for $k > 1$: we write $k(\delta)$ for the sequence which repeats $\delta$ $k$ times, and $\delta(k)$ for the sequence in which each number in $\delta$ is repeated $k$ times in place. We can also map between the groups $\Dim$ and $\mathsf{D}_{km}$ via the homomorphisms
$$\pi \colon \mathsf{D}_{km} \to \Dim \colon \quad f \mapsto f, \quad r \mapsto r, 
\qquad \textrm{and} \qquad
\iota \colon \Dim \to \mathsf{D}_{km} \colon \quad f \mapsto f, \quad r \mapsto r^k.$$

\begin{prop}
    If $\Gamma \subseteq \Dim$ is a reversal closed subset for $\delta \in \{1,-1\}^m$, then $\pi^{-1}(\Gamma) \subseteq \mathsf{D}_{km}$ is reversal closed for $k(\delta) \in \{1,-1\}^{km}$ and every maximal reversal-closed subset for $k(\delta)$ is of this form.
\end{prop}

\begin{proof}
    The homomorphism $\pi$ induces an action of $\mathsf{D}_{km}$ on $\{1,-1\}^m$, and we can check that the function $k(-) \colon \{1,-1\}^m \to \{1,-1\}^{km}$ is equivariant using the generators $r,f \in \mathsf{D}_{km}$:
    $$r(k(\delta)) = r(\delta\delta\cdots\delta) = r(\delta)r(\delta)\cdots r(\delta) \qquad\textrm{and}$$ $$f(k(\delta)) = f(\delta\delta\cdots\delta) = f(\delta)f(\delta)\cdots f(\delta)$$
    follow from the iterated repetition of $\delta$. As $k(-)$ is also a homomorphism of abelian groups, it preserves the type of equations that make a subset reversal-closed. Hence if
    $$\sigma_1(\delta) \cdot \sigma_2(\delta) \cdot \sigma_3(\delta) = \delta$$
    in $\{1,-1\}^m$ for $\sigma_1,\sigma_2,\sigma_3 \in \Dim$, then 
    $$\sigma'_1(\delta) \cdot \sigma'_2(\delta) \cdot \sigma'_3(\delta) = \delta$$
    whenever $\pi(\sigma'_i) = \sigma_i$ (i=1,2,3) and moreover 
    $$\sigma'_1(k(\delta)) \cdot \sigma'_2(k(\delta)) \cdot \sigma'_3(k(\delta)) = k(\delta)$$
    in $\{1,-1\}^{km}$, which suffices to show that $\pi^{-1}(\Gamma)$ is reversal-closed.

    That every maximal reversal-closed subset for $k(\delta)$ is of this form follows from the fact that in addition to being $\mathsf{D}_{km}$-equivariant and preserving group structure, the function $k(-) \colon \{1,-1\}^m \to \{1,-1\}^{km}$ is injective: if $\Gamma'$ is reversal closed for $k(\delta)$, the equations that demonstrate reversal-closure must hold for the $\mathsf{D}_{km}$-action on $\{1,-1\}^m$. This means that the same equations hold for the $\Dim$-action on the same set, where each element of $\Gamma'$ is replaced by its image under $\pi$ in $\Dim$, so $\pi(\Gamma')$ is reversal-closed for $\delta$ and hence $\pi^{-1}\pi(\Gamma')$ is reversal-closed for $k(\delta)$. As $\Gamma' \subseteq \pi^{-1}\pi(\Gamma')$, if $\Gamma'$ is maximal then $\Gamma' = \pi^{-1}\pi(\Gamma')$.
\end{proof}

\begin{ex}
    From \cref{trianglesubset} we have that $\{e,f,fr,rr\}$ and $\{e,f,r,frr\}$ are reversal-closed subsets for $(1,-1,1)$. As the homomorphism $\pi \colon \mathsf{D}_6 \to \mathsf{D}_3$ has kernel $\{e,rrr\}$, this means that for $2(1,-1,1) = (1,-1,1,1,-1,1)$ the subsets $\{e,f,fr,rr,rrr,frrr,frrrr,rrrrr\}$ and $\{e,f,r,frr,rrr,frrr,rrrr,frrrrr\}$ are reversal-closed and in fact the only maximal such subsets.
\end{ex}

\begin{ex}
    From \cref{squaresubset}, we can conclude that for any element of $\{1,-1\}^{4k}$ invariant under the rotation $r^4 \in \mathsf{D}_{4k}$ (and hence of the form $k(\delta)$ for some ${\delta \in \{1,-1\}^4}$), the entire $\mathsf{D}_{4k}$ is reversal-closed. This includes tiles such as $${(-1,1,1,1,-1,1,1,1) \in \mathsf{D}_8} \;\; \textrm{and} \;\;(1,-1,-1,1,1,-1,-1,1,1,-1,-1,1) \in \mathsf{D}_{12}.$$
\end{ex}

We can similarly describe reversal-closed subsets for $\delta(k)$.

\begin{prop}
    If $\Gamma$ is a reversal-closed subset for $\delta \in \{1,-1\}^m$, then $\iota(\Gamma)$ is reversal-closed for $\delta(k)$.
\end{prop}

Note that unlike the previous proposition, here the converse is not true and there may be maximal reversal-closed subsets for $\delta(k)$ not arising from those for $\delta$.

\begin{proof}
    For $\sigma \in \Dim$, it is straightforward to check that $\iota(\sigma)(\delta(k)) = \sigma(\delta)(k)$. As the function $-(k) \colon \{1,-1\}^m \to \{1,-1\}^{km}$ is additionally a group homomorphism, we have that whenever $\sigma_1(\delta) \cdot \sigma_2(\delta) \cdot \sigma_3(\delta) = \delta$, we can conclude that
    $$\iota(\sigma_1)(\delta(k)) \cdot \iota(\sigma_2)(\delta(k)) \cdot \iota(\sigma_3)(\delta(k)) = 
    \sigma_1(\delta)(k) \cdot \sigma_2(\delta)(k) \cdot \sigma_3(\delta)(k) = \delta(k).$$
    This means that the equations ensuring that $\iota(\Gamma)$ is reversal closed for $\delta(k)$ follow from those exhibiting $\Gamma$ as reversal-closed for $\delta$.
\end{proof}

\begin{ex}
    Among the many reversal-closed subsets this result allows us to construct, using \cref{squaresubset} we find that for any element $\delta(k) \in \{1,-1\}^{4k}$ for some $\delta \in \{1,-1\}^4$, $\iota(\mathsf{D}_4)$ is reversal-closed. We see this for instance for $(-1,-1,1,1,1,1,1,1)$, where $$\iota(\mathsf{D}_4) = \{e,f,rr,frr,rrrr,frrrr,rrrrrr,frrrrrr\} \subset \mathsf{D}_8$$
    is reversal-closed while all of the other maximal reversal-closed subsets have only 4 elements, at least one of which has an odd number of $r$'s:
    $\{e,frr,rrr,frrrrrrr\}$, $\{e,r,fr,frr\}$,
    $\{e,frr,frrr,rrrrrrr\}$, and $\{e,frr,rrrrr,frrrrr\}$.
\end{ex}

\end{document}